\documentclass[a4paper,11pt,oneside]{article}
\usepackage{amscd}
\usepackage{amsfonts}
\usepackage{amsmath}
\usepackage{amssymb}
\usepackage{amsthm}
\usepackage{array}
\usepackage{authblk}
\usepackage{caption}
\usepackage{cite}
\usepackage{colortbl}
\usepackage{ctable}
\usepackage{fancyhdr}
\usepackage{graphicx}
\usepackage[headings]{fullpage}
\usepackage{hyperref}
\usepackage{latexsym}
\usepackage{mathdots}
\usepackage{mathrsfs}
\usepackage{multirow}
\usepackage{rotating}
\usepackage{stmaryrd}
\usepackage{setspace}
\usepackage{tabularx}
\usepackage{titlesec}

\newcommand{\gt}[1]{\mathfrak{#1}}

\newcommand{\comment}[1]{}

\newcommand{\Aut}{\operatorname{Aut}}

\newcommand{\tw}{\mathrm{tw}}
\newcommand{\odd}{\operatorname{odd}}
\newcommand{\tr}{\operatorname{tr}}

\newcommand{\Cliff}{{\textsl{Cliff}}}
\newcommand{\CM}{{\text{CM}}}
\newcommand{\Weyl}{{\textsl{Weyl}}}
\newcommand{\WM}{{\text{WM}}}

\newcommand{\Div}{\operatorname{Div}}

\newcommand{\xmod}{{\rm \;mod\;}}

\newcommand{\SO}{\operatorname{\textsl{SO}}}    		
\newcommand{\Spin}{\operatorname{\textsl{Spin}}}		

\newcommand{\CC}{{\mathbb C}}
\newcommand{\PP}{{\mathbb P}}
\newcommand{\ZZ}{{\mathbb Z}}
\newcommand{\QQ}{{\mathbb Q}}
\newcommand{\lab}{{\langle}}    
\newcommand{\rab}{{\rangle}}    

\newcommand{\vv}{\mathbf v} 

\newcommand{\msn}{M^{s\natural}} 	

\renewcommand{\LL}{\Lambda}	
\newcommand{\Co}{\textsl{Co}}	

\newcommand{\KT}{{\rm K3}}
\newcommand{\rjg}{\varrho} 
\newcommand{\R}{{H}} 
\newcommand{\Hzw}{{\bf H}} 

\newcommand{\uA}{\mathrel{\raisebox{\depth}{\rotatebox{180}{$\!\!A$}}}}

\newtheorem{thm}{Theorem}[section]

\newtheorem{lem}[thm]{Lemma}
\newtheorem{pro}[thm]{Proposition}

\theoremstyle{definition}

\theoremstyle{remark}
\newtheorem{rmk}[thm]{Remark}

\numberwithin{equation}{section}

\begin{document}

\setstretch{1.26}

\title{
\vspace{-35pt}
    \textsc{\huge{{R}ational {K}3 {H}omotopy and the {L}argest {M}athieu {G}roup}}
    }

\renewcommand{\thefootnote}{\fnsymbol{footnote}} 

\renewcommand{\thefootnote}{\arabic{footnote}} 

\author[2,3,5]{Federico Carta\thanks{federico.carta@bo.infn.it}}
\author[1]{John F.\ R.\ Duncan\thanks{jduncan@as.edu.tw}}
\author[3,4]{Yang-Hui He\thanks{hey@maths.ox.ac.uk}}

\affil[1]{Institute of Mathematics, Academia Sinica, Taipei, Taiwan.}
\affil[2]{Istituto Nazionale Di Fisica Nucleare, Sezione di Bologna, Bologna, Italy.}
\affil[3]{London Institute for Mathematical Sciences, The Royal Institution, London, UK.}
\affil[4]{Merton College, University of Oxford, Oxford, UK.}
\affil[5]{Physics Department, King’s College London, London, UK.}

\date{} 

\maketitle

\abstract{
We interpret the ranks of the rational homotopy groups of a K3 surface as dimensions of representations for the largest sporadic simple Mathieu group.
We then construct a vertex algebra equipped with an action by the largest Mathieu group, and use it to associate Jacobi forms to this interpretation, in a compatible way.
Our results suggest a topological role for the sporadic simple Mathieu groups in the theory of K3 surfaces.
}

\clearpage

\tableofcontents

\section{Rational Homotopy}\label{sec:rh}

In the field of algebraic topology, homotopy theory plays a central role. 
Let $X$ be a compact path-connected topological space, 
and for $i\geq 1$ let $S^i$ be the $i$-dimensional sphere.
Then the {\em $i$-th homotopy group} of $X$, denoted $\pi_i(X)$, is the set of homotopy classes of base-point preserving maps from $S^i$ to $X$ (where we choose base points on $S^i$ and $X$ to make this definition, but these choices do not effect the isomorphism type of $\pi_i(X)$).

The first of these, $\pi_1(X)$,  
is the {\em fundamental group} of $X$, 
and it may be abelian or non-abelian, depending on $X$. 
However, the higher homotopy groups, $\pi_i(X)$ for $i>1$, 
are always abelian,
and generally infinite, being a direct product of a free part (i.e.\ $\mathbb{Z}$ to some power) and a torsion part (i.e.\ a product of finite cyclic groups). In general, the problem of calculating the homotopy groups 
of a given 
space, 
especially the torsion part, is 
complicated, and even for such spaces as $S^i$ it is still open.

In this paper we focus on 
the more accessible {\it rational} homotopy groups
$\pi_i(X)\otimes \QQ$.
Tensoring with the field of rationals has the effect of annihilating the torsion component, 
so that just the rank of the free part remains. 
For this reason nothing is lost if, for $j>0$, we focus on the {\em rational homotopy ranks}
\begin{gather}
\varrho_j(X) := \dim \pi_{j+1}(X)\otimes \QQ.
\end{gather}

Whilst 
the computation of the homotopy groups of the spheres is an important open problem as mentioned, the computation of their rational homotopy groups 
is a celebrated result of Serre 
\cite{MR0045386,MR0059548}:
\begin{gather}\label{eqn:rh:rhs}
\varrho_j(S^{b+1}) =
\begin{cases}
	1 &\text{ if $j=b$,}\\
	1 &\text{ if $j=2b$ and $b$ is odd,}\\
	0 &\text{ else.}
\end{cases}
\end{gather}

The theory of rational homotopy was initiated by works of Quillen \cite{MR0258031} and Sullivan \cite{MR0646078}, who independently offered algebraic approaches to computing rational homotopy groups. (See \cite{MR1721122} for a historical review.) 
In Sullivan's approach, which is based on de Rham cohomology, 
a commutative differential algebra called the 
minimal model of $X$ is associated to each path-connected space $X$. 
(See \cite{MR2355774} for an introduction to Sullivan's minimal models and see \cite{MR3741934} for a survey.)

In some cases the minimal model of a space may be computed directly from 
its cohomology ring. 
Such spaces are called {\em formal}. 
For example, 
smooth complex projective varieties are formal according to \cite{MR0382702}. 
It is generally still a challenge to compute the rational homotopy groups of smooth projective varieties, 
because their cohomology rings are 
generally beyond our control. 
But 
Babenko \cite{MR0552549}
(see also 
\cite{MR0578259,MR1632936}) 
was able to carry out the computation concretely for complete intersections (with ambient space $\PP^N$, for some $N$).
To state the result define 
\begin{gather}\label{eqn:ellX}
\ell(X):=(-1)^{n}(\chi-n-1),
\end{gather}
where $\chi=\chi(X)$ is 
the Euler characteristic of $X$, and $n=\dim(X)$ is the dimension of $X$. 

\begin{thm}[\!\!\cite{MR0552549}]\label{thm:babenko}
Let $X$ be a smooth complete intersection and set $\ell=\ell(X)$, $n=\dim(X)$ and $\chi=\chi(X)$.
        If $\chi = n+1$ then
        \begin{gather}
        \label{eqn:varrhoell=0}
\varrho_j(X)
        = 
        \begin{cases}
        1&\text{ if $j=1$ or $j=2n$,}\\
        0&\text{ else,}
        \end{cases}
        \end{gather}
whereas if $\chi \neq n+1$ then 
        \begin{equation}\label{eqn:varrhoellnot0}
        \varrho_j(X)
        =
        \frac{(-1)^{j}}{j}
        \sum\limits_{k | j}
        (-1)^k \mu\left( \frac{j}{k} \right) \sum\limits_{i=1}^{2n-2} \xi_i^{-k}    ,
        \end{equation}
        where the first summation is over the divisors of $j$, and $\xi_{1},\ldots,\xi_{2n-2}$, together with $\xi_{2n-1} = -1$, are the roots of the polynomial
        \begin{gather}\label{eqn:rootsofpolynomial}
        1 -\ell  z^{n-1} - \ell  z^{n} + z^{2n-1} 
        .
        \end{gather}
\end{thm}

\begin{rmk}
It is known 
    (see e.g.\ \cite{MR0137706})
    that the Euler number $\chi=\chi(X)$ of a smooth 
    complete intersection $X$ as in Theorem \ref{thm:babenko} is the coefficient of $x^n$ in the expansion about $x=0$ of
    \begin{gather}\label{eqn:chi}
    \frac{1}{(1-x)^2} \prod\limits_{i=1}^r
    \frac{a_i}{1 + (a_i - 1) x},
    \end{gather}
    where $n=\dim(X)$, and $a_1,\dots,a_r$ denote the degrees of the hypersurfaces that define $X$.
    In the special case that $r=1$, so that $X$ is a hypersurface of degree $d = a_1$, we obtain from this that
    \begin{gather}\label{eqn:chiHyper}
    \chi = \frac{1}{d} \left(
    (1-d)^{n+2} + d(n+2) - 1
    \right).        
    \end{gather}
\end{rmk}

\begin{rmk}
For a simple test of Theorem \ref{thm:babenko} consider a line $X$ in $\PP^2$, which is a copy of $\PP^1$ and topologically $S^2$.
Since $r=1$ we may apply (\ref{eqn:chiHyper}) with $d=n=1$ to obtain $\chi =2$. 
Thus $\chi=n+1$, and by
the first part of Theorem \ref{thm:babenko} we expect that the rational homotopy group $\pi_{i}(X)\otimes \QQ$ will be $\QQ$ for $i=2$ and $i=3$, and trivial otherwise. 
This agrees precisely with Serre's result (\ref{eqn:rh:rhs}).
\end{rmk}

Babenko's proof of (\ref{eqn:varrhoell=0}-\ref{eqn:varrhoellnot0}) is mostly concerned with the cohomology of the loop space $\Omega X$ of $X$. This is because we have
the plethystic formula
\begin{gather}\label{eqn:rh:POmegaVple}
	P_{\Omega X}(x)^{-1} = \prod_{j>0} (1-(-x)^j)^{(-1)^{j}\varrho_j(X)}
\end{gather}
(cf.\ \cite{MR2313959})
according to Lemma 2 of \cite{MR0578259}, where 
\begin{gather}\label{eqn:PX}
	P_X(x):=\sum_{i\geq 0}\dim H^i(X,\QQ)x^i
\end{gather}
is the {\em Poincar\'e series} of $X$. 

\begin{rmk}\label{rmk:plethysm-mult}
We pause here to emphasize that, since plethysm (\ref{eqn:rh:POmegaVple}) is multiplicative, 
$P_{\Omega X}$ and its reciprocal $P_{\Omega X}^{-1}$ are as good as each other when it comes to computing $\varrho_j(X)$ in the setup of \cite{MR0552549,MR0578259}.
\end{rmk}

The significance of the definition (\ref{eqn:ellX}) is that if $X$ is a smooth complete intersection of dimension $n$, then $\ell(X)$ is the number of generators in $H^n(X,\QQ)$ of the cohomology ring $H^*(X,\QQ)$ (see Proposition 2.1 of \cite{MR0552549}). 
Also, 
Theorem 1 of \cite{MR0552549} tell us that
\begin{gather}\label{eqn:POmegaVell0}
	P_{\Omega X}(x) = \frac{1+x}{1-x^{2n}}
\end{gather}
if 
$\ell(X)=0$, 
whereas if
$\ell(X)\neq 0$ then
\begin{gather}\label{eqn:POmegaVellnot0}
	P_{\Omega X}(x) = 
	\frac
	{1+x}
	{1-\ell x^{n-1}-\ell x^{n}+x^{2n-1}},
\end{gather}
where $\ell=\ell(X)$.
So, in light of (\ref{eqn:rh:POmegaVple}), the identities (\ref{eqn:varrhoell=0}) and (\ref{eqn:varrhoellnot0}) follow from (\ref{eqn:POmegaVell0}) and (\ref{eqn:POmegaVellnot0}), respectively, 
and
the $\xi_i$ that appear in (\ref{eqn:varrhoellnot0}) are alternatively characterised as the poles of the Poincar\'e series (\ref{eqn:POmegaVellnot0}) of the loop space of $X$ (when $\ell(X)\neq 0$).
\begin{rmk}
According to Proposition 2.1 of \cite{MR0552549} we have
\begin{gather}\label{eqn:rh:PVx}
	P_X(x) = 
	\frac
	{1+\ell x^n - \ell x^{n+2} - x^{2n+2}}
	{1-x^2}
\end{gather}
for the Poincar\'e series of a smooth complete intersection $X$ as in Theorem \ref{thm:babenko}, 
where $\ell=\ell(X)$ is as in (\ref{eqn:ellX}), and $n=\dim(X)$, and this formula (\ref{eqn:rh:PVx}) holds also when $\ell(X)=0$.
\end{rmk}

Now consider the case that $X$ is a smooth quartic in $\PP^3$, i.e.\ a K3 surface. 
Then $n=\dim(X)=2$, and according to (\ref{eqn:chiHyper}) we have $\chi(\KT)=\chi(X)=24$. 
Thus
 (\ref{eqn:ellX}) and (\ref{eqn:POmegaVellnot0})
yield
\begin{gather}\label{eqn:rh:POmegaK3}
	P_{\Omega \KT} (x )^{-1} =
	\frac
	{1-21x - 21x^2+x^3} 
	{1+x}
	=
	{1-22x+x^2}
	.
\end{gather}
Now, using (\ref{eqn:varrhoellnot0}) or the plethystic expansion (\ref{eqn:rh:POmegaVple}), 
which in this case reads as
\begin{gather}\label{eqn:rh:POmegaVple-2}
\begin{split}
	P_{\Omega\KT}(x)^{-1} 
	= 
	{1-22x+x^2}
	&= \prod_{j>0} (1-(-x)^j)^{(-1)^{j}\varrho_j(\KT)}
	\\
	&=\prod_{j>0}
	\frac
	{(1-x^{2j})^{\varrho_{2j}(\KT)}}
	{(1+x^{2j-1})^{\varrho_{2j-1}(\KT)}}
	,
\end{split}
\end{gather}
we may compute
\begin{gather}\label{eqn:dimpiiK3}
 	\sum_{j>0}
    \varrho_j(\KT)x^j = 
    22x
    + 252x^2
    + 3520x^3
    + 57960x^4 
    + 1020096x^5 
    + \dots.
\end{gather}

Comparing (\ref{eqn:dimpiiK3}) with the character table of the sporadic simple group $M_{24}$ (which we reproduce in Table \ref{tab:ct}) we note that $\rjg_2(\KT)=252$ and $\rjg_3(\KT)=3520$ are dimensions of irreducible representations of $M_{24}$, 
while $\rjg_1(\KT)=22$ is just $1$ less than the minimal dimension of a non-trivial representation of 
$M_{24}$.
In the next section we will explain how to 
interpret all the K3 rational homotopy ranks 
$\varrho_j(\KT)$ for $j>1$ as dimensions of representations of $M_{24}$.
		
\section{The Largest Mathieu Group}\label{sec:m24}

The {\em largest Mathieu group}, $M_{24}$, is the unique $5$-transitive group of permutations on $24$ points that is not the full symmetric group $S_{24}$, or its subgroup $A_{24}$.
(See \cite{MR1409812}.)
It was discovered by \'Emile Mathieu (see \cite{Mat_1861,Mat_1873}) and its order is
\begin{gather}\label{eqn:M24}
	\# M_{24} = 2^{10}\cdot 3^3\cdot 5\cdot 7\cdot 11\cdot 23 = 244823040.
\end{gather}
It is one of the sporadic simple groups.
The stabilizer of a point in $M_{24}$ is denoted $M_{23}$, is also a sporadic simple group, and satisfies
\begin{gather}\label{eqn:M23}
	\# M_{23} = 2^{7}\cdot 3^2\cdot 5\cdot 7\cdot 11\cdot 23 = 10200960.
\end{gather}

For $g\in M_{24}$ set $\rjg_1'(g):=\chi_2(g)-\chi_1(g)$, where $\chi_j$, for $1\leq j\leq 26$, denotes an irreducible character of $M_{24}$ as specified in Table \ref{tab:ct}. 
Next define
\begin{gather}\label{eqn:rh:Pg}
	P_g(x):=\frac{1}{1-\rjg_1'(g)x+x^2},
\end{gather}
so that $P_e=P_{\Omega\KT}$ (cf.\ (\ref{eqn:rh:POmegaK3})) for $e$ the identity element of $M_{24}$.
Then, by 
taking the logarithm of (\ref{eqn:rh:POmegaVple-2}) we obtain
\begin{gather}\label{eqn:rh:logPe}
\begin{split}
	\log P_e(x)^{-1} 
	&
	=  
	\sum_{j>0} (-1)^{j}\varrho_j(\KT) \log (1-(-x)^j)
	\\
	&
	=  
	-\sum_{j,k>0} (-1)^{j(k+1)}{\varrho_j(\KT)}\frac1{k} {x^{jk}}.
\end{split}
\end{gather}
We now define $\rjg_j(g)$, for $j>0$ and $g\in M_{24}$, by requiring that
\begin{gather}\label{eqn:rh:def-rhojg}
	\log P_g(x)^{-1}
	=  
	-\sum_{j,k>0} (-1)^{j(k+1)}{\rjg_j(g^k)}\frac1k {x^{jk}}.
\end{gather}

The idea here (\ref{eqn:rh:def-rhojg}) is that $\rjg_j(g)$ should be the trace of $g\in M_{24}$ on an $M_{24}$-module $R_j$ with 
\begin{gather}\label{eqn:dimRj}
	\dim R_j = \rjg_j(\KT).
\end{gather}
But it is perhaps not clear even that the values $\rjg_j(g)$ are all well-defined. 
In fact we have the following theorem, which is our first main result, and which interprets all the rational K3 homotopy ranks $\varrho_j(\KT)$ for $j>1$ as dimensions of representations of $M_{24}$. 
\begin{thm}\label{thm:rh:rhoj}
We have $\rjg_1=\rjg'_1$ and $\rjg_j(e)=\varrho_j(\KT)$, and $\rjg_j$ is the character of a representation $R_j$ of  $M_{24}$ for $j>1$.
In particular, $\rjg_j(g)$ is well-defined for $j>0$ and $g\in M_{24}$.
\end{thm}
\begin{proof}
We obtain the identity $\rjg_1(g)=\rjg_1'(g)$ by 
recalling the definition (\ref{eqn:rh:Pg}) of $P_g$ and 
reducing (\ref{eqn:rh:def-rhojg}) modulo $x^2$. 
We obtain $\rjg_j(e)=\varrho_j(\KT)$ by taking $g=e$ in (\ref{eqn:rh:def-rhojg}) and comparing with (\ref{eqn:rh:logPe}).
Thus $\rjg_j(g)$ is well-defined, and integer-valued, at least when $j=1$ or $g=e$.

To see that $\rjg_j(g)$ is well-defined and integer-valued in general 
we apply Lemma \ref{lem:repthyple} to the reciprocal of $f = 1- U_1x+x^2$, where 
$U_1$ represents a (virtual) module with character $\rjg_1=\rjg_1'$. 

For $\chi$ an irreducible character of $M_{24}$ let $\rjg_j(\chi)$ denote the multiplicity of $\chi$ in the character $g\mapsto \rjg_j(g)$. 
It remains to show that the $\rjg_j(\chi)$ are non-negative for $j>1$. 
We computed directly that $\rjg_j(\chi)\geq 0$ for all irreducible $\chi$, for $1<j < 20$, using Mathematica \cite{Mathematica}. See the tables in \S~\ref{sec:mlt} for the computed values.
In the remainder we will explain how to show that the $\rjg_j(\chi)$ are non-negative for $j\geq 16$.

For $g\in M_{24}$ let $C(g)$ denote the centralizer of $g$ in $M_{24}$. 
According to the character theory of finite groups we have 
\begin{gather}\label{eqn:rhomchi}
	\rjg_j(\chi) = \sum_{[g]}\frac{\rjg_j(g)}{\# C(g)} \overline{\chi(g)},
\end{gather}
where the sum is over the conjugacy classes of $M_{24}$.
Since $\rjg_j(e)> |\rjg_j(g)|$ for $g\neq e$ in the case that $\rjg_j$ is a character, we expect the main term in (\ref{eqn:rhomchi}) to be the one corresponding to $[g]=[e]$. 
Applying the triangle inequality to (\ref{eqn:rhomchi}) we obtain
\begin{gather}\label{eqn:rhomchi_ineq}
	\rjg_j(\chi) \geq \frac{\rjg_j(e)}{\# M_{24}} \dim(\chi) - \sum_{[g]\neq [e]}\frac{|\rjg_j(g)|}{\# C(g)} |{\chi(g)}|,
\end{gather}
so it suffices to show that
\begin{gather}\label{eqn:rhomchi_ineq2}
	\frac{\rjg_j(e)}{\# M_{24}} \dim(\chi) \geq \sum_{[g]\neq [e]}\frac{|\rjg_j(g)|}{\# C(g)} |{\chi(g)}|,
\end{gather}
for all irreducible $\chi$, for $j\geq 16$. Since $\dim(\chi)\geq |\chi(g)|$ for all $g\in M_{24}$, and $\#M_{24}=\#C(g) \#[g]$, 
we may replace (\ref{eqn:rhomchi_ineq2}) with the more crude inequality
\begin{gather}\label{eqn:rhom_ineq3}
	{\rjg_j(e)} \geq \sum_{[g]\neq [e]}\#[g]{|\rjg_j(g)|}.
\end{gather}

Suppose now that $A(x)$ and $B(x)$ are functions of a real variable $x$ with the property that, for $j\geq 16$, we have
$\rjg_j(e)>A(j)$ and $B(j)\geq |\rjg_j(g)|$ for all $g\neq e$. Then for $j\geq 16$ the right-hand side of (\ref{eqn:rhom_ineq3}) is bounded above by $\#M_{24}B(j)$, and the required identity will follow so long as 
\begin{gather}\label{eqn:ABineq}
A(x)>\#M_{24}B(x)
\end{gather} 
for $x\geq 16$. 
Now define
\begin{gather}
\begin{split}
	A(x) &:= 
	\frac1x 22^x\left(\frac{481^x}{482^x}+\frac{1}{483^x}\right) - 22^{\frac x 2},\\
	B(x) &:=
	\frac1x 6^x
	+ \frac1x 22^{\frac{x}2} 
	+6^{\frac{x}3} + 22^{\frac{x}{6}}.
\end{split}
\end{gather}
Then Lemma \ref{lem:rh:rhom1A-lb} shows that $\rjg_j(e)>A(j)$ for $j\geq 16$, and Lemma \ref{lem:rh:rhomg-ub} shows that $B(j)\geq |\rjg_j(g)|$ for $g\neq e$ and $j\geq 16$.
The inequality (\ref{eqn:ABineq}) follows, for $x\geq 16$, from a direct check.
This completes the proof.
\end{proof}
 
\section{Vertex Algebra and Jacobi Forms}\label{sec:va} 
 
For $\tau\in\CC$ with $\Im(\tau)>0$ the {\em Dedekind eta function} is 
$\eta(\tau):=q^{\frac1{24}}\prod_{n>0}(1-q^n)$,
where $q=e^{2\pi i \tau}$,
and the {\em Jacobi theta functions} are 
\begin{gather}	
\label{eqn:sf:jac}
\theta_1(\tau,z)
	:= -i q^{\frac18} y^{\frac12} \prod_{n>0}  (1-y q^n) (1-y^{-1} q^{n-1})(1-q^n),\\
	\theta_2(\tau,z)
	:=  q^{\frac18} y^{\frac12} \prod_{n>0}  (1+y q^n) (1+y^{-1} q^{n-1})(1-q^n),\\
	\theta_3(\tau,z)
	:=  \prod_{n>0} (1+y \,q^{n-\frac12}) (1+y^{-1} q^{n-\frac12}) (1-q^n),\\
	\theta_4(\tau,z) 
	:=  \prod_{n>0}  (1-y \,q^{n-\frac12}) (1-y^{-1} q^{n-\frac12})(1-q^n),
\end{gather}
where $q=e^{2\pi i \tau}$ and $y=e^{2\pi i z}$.
From these concrete specifications it follows 
that the function
\begin{gather}\label{eqn:va:phi}
	\R(\tau,z)
	:=
	-\frac12\frac{\theta_4(\tau,2z)}{\theta_4(\tau,0)}\frac{\eta(\frac{\tau}{2})^{24}}{\eta(\tau)^{24}}
	+\frac12\frac{\theta_3(\tau,2z)}{\theta_3(\tau,0)}\frac{\eta({\tau})^{48}}{\eta(\frac{\tau}{2})^{24}\eta(2\tau)^{24}}
	-\frac12\frac{\theta_2(\tau,2z)}{\theta_2(\tau,0)}2^{12}\frac{\eta({2\tau})^{24}}{\eta(\tau)^{24}}
\end{gather}
is a weak Jacobi form of weight $0$ and index $2$.
(See \cite{MR781735} for background on Jacobi forms.) 
Moreover, by explicit calculation 
we may compute that
\begin{gather}\label{eqn:va:limImtauphi}
	\R(\tau,z) = y^{-2}+22+y^{2} +O(q) 
\end{gather}
as $\Im(\tau)\to 0$. 
Thus
we recover $P_{\Omega\KT}^{-1}$ (recall (\ref{eqn:rh:POmegaK3}-\ref{eqn:rh:POmegaVple-2})) from $\R$ via the formula
\begin{gather}\label{eqn:va:phitoPOmegaK3}
	\lim_{\Im(\tau)\to \infty} \R(\tau,z)
	=
	y^{-2}P_{\Omega\KT}(-y^2)^{-1}
	.
\end{gather}
In other words, 
we may regard $\R$ as a modular extension of $P_{\Omega\KT}^{-1}$.

We have associated a twining $P_g$ (see (\ref{eqn:rh:Pg})) of $P_{\Omega\KT}$ to each element $g$ of the Mathieu group $M_{24}$.
In this section we will use a vertex-algebraic construction to attach a weak Jacobi form $\R_g$ to each $g\in M_{24}$ 
in such a way that
\begin{enumerate}
\item
We recover the 
$P_g^{-1}$ (recall (\ref{eqn:rh:logPe}-\ref{eqn:rh:def-rhojg}))
from the $\R_g$ 
via the formula
\begin{gather}\label{eqn:va:phigtoPg}
	\lim_{\Im(\tau)\to \infty} \R_g(\tau,z)
	=
	y^{-2}P_{g}(-y^2)^{-1},
\end{gather}
and
\item
The $\R_g$ are the graded traces defined by 
a bigraded infinite-dimensional module for $M_{24}$.
\end{enumerate}

Before proceeding we mention that the function $\R(\tau,z)$ of (\ref{eqn:va:phi}) appeared earlier in connection with Mathieu groups in \cite{MR3373711}. 
In that work a vertex algebra with an action of $M_{23}$ (cf.\ (\ref{eqn:M23})) is constructed, via a method very similar to that of \cite{MR3465528}, and used to associate a Jacobi form $Z_g^{\mathcal{N}=2}$ of weight $0$ and index $2$ (with level) to each $g\in M_{23}$. 
In fact our construction agrees with theirs in that we have $\R_g=Z_g^{\mathcal{N}=2}$ for each $g\in M_{23}$.
However, whilst the construction of \cite{MR3373711} does not extend to $M_{24}$, 
the technique we use here (which is similar to that used in \cite{MR3922534}) allows us to define $\R_g$ for all $g\in M_{24}$.

To prepare for what follows we agree to use the term {\em symplectic (vector) space} to refer to a vector space equipped with a non-degenerate alternating bilinear form, and we use the term {\em orthogonal (vector) space} to mean a vector space equipped with a non-degenerate symmetric bilinear form. 
 
There are standard constructions which attach a super vertex algebra, and a canonically twisted module for it, to 
an orthogonal or symplectic vector space.
This is reviewed e.g.\ in Section 2 of \cite{MR3859972}, and we adopt the notation of that reference here.
In particular, 
we write $A(\gt{a})$ and 
$A(\gt{a})_\tw$ (respectively) for 
the super vertex algebra and 
canonically twisted $A(\gt{a})$-module attached by the construction of Section 2.1 of \cite{MR3859972} to a 
complex orthogonal space $\gt{a}$, 
and write $\uA(\gt{b})$ and
$\uA(\gt{b})_\tw$ (respectively) for  
the super vertex algebra and 
canonically twisted $\uA(\gt{b})$-module 
attached by the construction of Section 2.2 of \cite{MR3859972} to a 
complex symplectic space $\gt{b}$. 

Henceforth let $\gt{a}$ denote a $2$-dimensional orthogonal space, 
let $\gt{b}$ be a $2$-dimensional symplectic space, 
let $\gt{v}$ be a $24$-dimensional orthogonal space,
and define
\begin{gather}\label{eqn:va:W}
	M:=A(\gt{a})\otimes \uA(\gt{b})\otimes A(\gt{v}),\quad
	M_\tw:=A(\gt{a})_\tw\otimes \uA(\gt{b})_\tw\otimes A(\gt{v})_\tw.
\end{gather}
Then $M$ is naturally a super vertex algebra, and $M_\tw$ is naturally a canonically twisted module for $M$.

We wish to equip $M$ and $M_\tw$ with compatible actions of the Mathieu group $M_{24}$. 
For this we recall that (according to the conventions of \cite{MR3859972}) the {\em Clifford algebra} attached to $\gt{v}$ is the quotient
\begin{gather}\label{eqn:Cliffgtv}
	\Cliff(\gt{v}):=T(\gt{v})/\lab v\otimes v'+v'\otimes v - \lab v,v'\rab {\bf 1}\mid v,v'\in\gt{v}\rab
\end{gather}
of the tensor algebra $T(\gt{v})$ of $\gt{v}$ by the ideal generated by expressions of the form 
$v\otimes v'+v'\otimes v - \lab v,v'\rab {\bf 1}$ for $v,v'\in\gt{v}$.
This is useful because it allows us to concretely identify the {\em Spin group} attached to $\gt{v}$ with the set
\begin{gather}
	\Spin(\gt{v}) := 
	\{
	x\in \Cliff(\gt{v})^* \mid \alpha(x)x={\bf 1}
	 \}
\end{gather}
of invertible elements of $\Cliff(\gt{v})$ such that $\alpha(x)$ is the inverse of $x$, where $\alpha$ is the unique anti-automorphism of $\Cliff(\gt{v})$ that satisfies $\alpha(v_1\dots v_k)=v_k\dots v_1$ for $v_j\in \gt{v}$.
Alternatively, $\Spin(\gt{v})$ is generated by the exponentials of the expressions of the form $vv'-v'v$ for $v,v'\in\gt{v}$.
For example, if $v^\pm\in\gt{v}$ are isotropic vectors such that $\lab v^-,v^+\rab=1$, then 
\begin{gather}\label{eqn:va:X}
X=i(v^+v^- -v^-v^+)
\end{gather}
satisfies $X^2=-{\bf 1}$ in $\Cliff(\gt{v})$, 
so $e^{a X}=\cos(a){\bf 1}+\sin(a)X$ in $\Cliff(\gt{v})$ for $\nu\in \CC$, and the exponential $e^{a X}$ belongs to $\Spin(\gt{v})$.

For $x\in \Spin(\gt{v})$ and $v\in \gt{v}$ define
\begin{gather}\label{eqn:x(v)}
	x(v):=xvx^{-1},
\end{gather}
where the product on the right-hand side of (\ref{eqn:x(v)}) is evaluated in $\Cliff(\gt{v})$. 
Then $x(v)$ belongs to the canonical copy of $\gt{v}$ inside $\Cliff(\gt{v})$, and the assignment $x\mapsto x(\cdot )$ defines a surjective map $\Spin(\gt{v})\to \SO(\gt{v})$ with exactly $\{\pm{\bf 1}\}$ as its kernel.

The group $\Spin(\gt{v})$ acts naturally on $A(\gt{v})$ and $A(\gt{v})_\tw$ in such a way that $xY_\tw(v,z)v' = Y_{\tw}(xv,z)xv'$ for $x\in \Spin(\gt{v})$ and $v\in A(\gt{v})$ and $v'\in A(\gt{v})_\tw$. 
Indeed, the natural action of $\Spin(\gt{v})$ on $A(\gt{v})$ is characterized by the requirement that 
\begin{gather}\label{eqn:SpingtvonAgtv}
	xv_1(-n_1+\tfrac12)\dots v_k(-n_k+\tfrac12){\bf v}
	=x(v_1)(-n_1+\tfrac12)\dots x(v_k)(-n_k+\tfrac12){\bf v}
\end{gather}
(cf.\ (\ref{eqn:x(v)})) for $v_j\in \gt{v}$ and $n_j\leq 0$. Note that this action (\ref{eqn:SpingtvonAgtv}) factors through to the special orthogonal group $\SO(\gt{v})$ of $\gt{v}$.
To define the action of $\Spin(\gt{v})$ on $A(\gt{v})_\tw$ we 
recall (see e.g.\ \S~2.1 of \cite{MR3859972}) that there is a natural identification 
\begin{gather}\label{eqn:Agtvtw_simeq_wedgeCMgtv}
A(\gt{v})_\tw
\simeq 
\bigwedge(v(-n)\mid v\in\gt{v},\;n> 0)\otimes \CM(\gt{v}),
\end{gather}
where $\CM(\gt{v})$ denotes (a realization of, cf.\ (\ref{eqn:CMgtv=Cliffgtvetc})) the unique irreducible module for $\Cliff(\gt{v})$.
The action of $\Spin(\gt{v})$ on $A(\gt{v})_\tw$ is then given by
\begin{gather}\label{eqn:SpingtvonAgtvtw}
	xv_1(-n_1)\dots v_k(-n_k)\otimes y
	=x(v_1)(-n_1)\dots x(v_k)(-n_k)\otimes xy,
\end{gather}
for $v_j\in \gt{v}$ and $n_j< 0$ and $y\in \CM(\gt{v})$.

Our next step is to realize $M_{24}$ as a subgroup of $\SO(\gt{v})$ by choosing an orthonormal basis for $\gt{v}$ and letting $M_{24}$ act as permutations on these basis vectors. 
Let us suppose we have done this, and write $G$ for the subgroup of $\SO(\gt{v})$ so constructed.
Then the argument of Proposition 3.1 of \cite{MR3465528} (with $M_{24}$ in place $\Co_0$) of shows that there is a subgroup of $\Spin(\gt{v)}$ that is both isomorphic to $G$ and contained in the preimage of $G$ under the natural map $\Spin(\gt{v})\to \SO(\gt{v})$. 
We now have a copy of $M_{24}$ in $\Spin(\gt{v})$, and therefore also actions of $M_{24}$ on $M$ and $M_\tw$.
 
The representation of $M_{24}$ that 
we use to define the $\R_g$ is constructed from $M$ and $M_\tw$. 
To proceed we
let $\vv$ denote the vacuum of $M$, let $\gt{z}_{\gt{a}}$ denote the canonical involution on $A(\gt{a})$, and interpret the symbols $\gt{z}_\gt{b}$ and $\gt{z}_\gt{v}$ similarly. 
Once and for all we choose decompositions $\gt{a}=\gt{a}^-\oplus \gt{a}^+$ and $\gt{b}=\gt{b}^-\oplus \gt{b}^+$ and $\gt{v}=\gt{v}^-\oplus \gt{v}^+$, of $\gt{a}$ and $\gt{b}$ and $\gt{v}$ into maximal isotropic subspaces. 
Then there
is a unique (up to scale) vector $\vv_{\tw,\gt{v}}\in A(\gt{v})_\tw$ with the property that
\begin{gather}\label{eqn:va:anvvtwgta}
	v(n)\vv_{\tw,\gt{v}}=0
\end{gather}
when either $n<0$, or $n=0$ and $v\in \gt{v}^+$. 
Indeed, 
we may realize $\CM(\gt{v})$ explicitly as 
\begin{gather}\label{eqn:CMgtv=Cliffgtvetc}
	\CM(\gt{v}) = \Cliff(\gt{v})\otimes_{\lab \gt{v}^+\rab}\CC\vv_{\tw,\gt{v}},
\end{gather}
where $\lab\gt{v}^+\rab$ denotes the subalgebra of $\Cliff(\gt{v})$ generated by $\gt{v}^+$, and $\CC\vv_{\tw,\gt{v}}$ denotes the $1$-dimensional $\lab\gt{v}^+\rab$-module characterized by the conditions that ${\bf 1}\vv_{\tw,\gt{v}}=\vv_{\tw,\gt{v}}$ and $\gt{v}^+\vv_{\tw,\gt{v}}\subset \{0\}$.
Then the $\vv_{\tw,\gt{v}}$ in (\ref{eqn:va:anvvtwgta}) is ${\bf 1}\otimes \vv_{\tw,\gt{v}}$ in (\ref{eqn:CMgtv=Cliffgtvetc}).
We let $\vv_{\tw,\gt{a}}$ denote the vector in $A(\gt{a})_\tw$ that is characterized (up to scale) in the directly analogous way.

There is also a counterpart $\vv_{\tw,\gt{b}}\in \;\uA(\gt{b})_\tw$ to $\vv_{\tw,\gt{a}}$ and $\vv_{\tw,\gt{v}}$, but to identify it we need 
the irreducible module 
\begin{gather}\label{eqn:WMgtb=Weylgtbetc}
	\WM(\gt{b}) = \Weyl(\gt{b})\otimes_{\lab \gt{b}^+\rab}\CC\vv_{\tw,\gt{b}}
\end{gather}
(cf.\ (\ref{eqn:CMgtv=Cliffgtvetc}))
for the Weyl algebra
\begin{gather}\label{eqn:Weylgtb}
	\Weyl(\gt{b}):=T(\gt{b})/\lab b\otimes b'-b'\otimes b - \lab b,b'\rab {\bf 1}\mid b,b'\in\gt{b}\rab
\end{gather}
(cf. (\ref{eqn:Cliffgtv})).
In (\ref{eqn:WMgtb=Weylgtbetc}) we use
$\lab\gt{b}^+\rab$ to denote the subalgebra of $\Weyl(\gt{b})$ generated by $\gt{b}^+$, and $\CC\vv_{\tw,\gt{b}}$ is the $1$-dimensional $\lab\gt{b}^+\rab$-module characterized by the conditions that ${\bf 1}\vv_{\tw,\gt{b}}=\vv_{\tw,\gt{b}}$ and $\gt{b}^+\vv_{\tw,\gt{b}}\subset \{0\}$.
With these definitions we have
\begin{gather}
\uA(\gt{b})_\tw
\simeq 
\bigvee(b(-n)\mid b\in\gt{b},\;n> 0)\otimes \WM(\gt{b}),
\end{gather}
and we use
$\vv_{\tw,\gt{b}}$ as a shorthand for ${\bf 1}\otimes \vv_{\tw,\gt{b}}$.

We extend the action of $\gt{z}_\gt{a}$ to $A(\gt{a})_\tw$ by requiring that $\gt{z}_\gt{a}\vv_{\tw,\gt{a}}=i\vv_{\tw,\gt{a}}$ and also that
\begin{gather}
	Y_\tw(\gt{z}_\gt{a}v,z)\gt{z}_\gt{a}v' = \gt{z}_\gt{a} Y_\tw(v,z)v'
\end{gather}
for $v\in A(\gt{a})$ and $v'\in A(\gt{a})_\tw$. 
We extend the action of $\gt{z}_\gt{b}$ to $\uA(\gt{b})_\tw$ by requiring that $\gt{z}_\gt{b}\vv_{\tw,\gt{b}}=-i\vv_{\tw,\gt{b}}$ and also that
\begin{gather}
	Y_\tw(\gt{z}_\gt{b}v,z)\gt{z}_\gt{b}v' = \gt{z}_\gt{b} Y_\tw(v,z)v'
\end{gather}
for $v\in\; \uA(\gt{b})$ and $v'\in\; \uA(\gt{b})_\tw$. 
For the action of $\gt{z}_\gt{v}$ on $A(\gt{v})_\tw$ we require simply that $\gt{z}_\gt{v}\vv_{\tw,\gt{v}}=\vv_{\tw,\gt{v}}$, and also
\begin{gather}
	Y_\tw(\gt{z}_\gt{v}v,z)\gt{z}_\gt{v}v' = \gt{z}_\gt{v} Y_\tw(v,z)v'
\end{gather}
for $v\in A(\gt{v})$ and $v'\in A(\gt{v})_\tw$.

Set $\vv_\tw:=\vv_{\tw,\gt{a}}\otimes  \vv_{\tw,\gt{b}}\otimes \vv_{\tw,\gt{v}}$ and set $\gt{z}:=\gt{z}_\gt{a}\otimes \gt{z}_\gt{b}\otimes \gt{z}_\gt{v}$, and write $M^j$ and $M_\tw^j$ for the $(-1)^j$ eigenspaces 
for the action of $\gt{z}$ on $M$ and $M_\tw$ (respectively). 
Then, as explained in \cite{Dun_VACo,MR3376736}, the method of \cite{MR1372720} defines a super vertex algebra structure on
\begin{gather}\label{eqn:va:Msnatural}
	M^{s\natural}:=M^0\oplus M_\tw^1,
\end{gather}
and naturally equips the space 
\begin{gather}\label{eqn:va:Msnaturaltw}
	M^{s\natural}_\tw:=M^1\oplus M_\tw^0
\end{gather}
with the structure of a canonically twisted module for $M^{s\natural}$.
Also, the actions of $M_{24}$ on $M$ and $M_{\tw}$ commute with $\gt{z}$, and so $M_{24}$ acts naturally on $M^{s\natural}$ and $M^{s\natural}_\tw$.

We are ready to define the $\R_g$. For this  
we 
choose, for each $g\in M_{24}$, a basis $\{v_{g,j}^\pm\}_{j=1}^{12}$ for 
$\gt{v}$ such that each $v_{g,j}^\pm$ is an isotropic eigenvector for $g$, and 
\begin{gather}\label{eqn:vgjmpvgkpm}
	\lab v_{g,j}^\mp,v_{g,k}^\pm\rab=\delta_{jk}.
\end{gather}
Note that, after swapping e.g.\ $v_{g,1}^+$ with $v_{g,1}^-$ if necessary, we may assume 
that the unique (up to scale) vector $\vv_{\tw,\gt{v},g}$ in 
$\CM(\gt{v})$ (cf.\ (\ref{eqn:Agtvtw_simeq_wedgeCMgtv}) and (\ref{eqn:va:anvvtwgta}-\ref{eqn:CMgtv=Cliffgtvetc})) such that
\begin{gather}
v_{g,j}^+(0)\vv_{\tw,\gt{v},g}=0
\end{gather} 
for all $j$, is fixed by $\gt{z}_\gt{v}$. This ensures that 
$\vv_{\tw,\gt{a}}\otimes\vv_{\tw,\gt{b}}\otimes \vv_{\tw,\gt{v},g}$  belongs to $\msn_\tw$ (cf.\ \ref{eqn:va:Msnaturaltw}).

Next let $K_g(0)$ denote the operator on $\msn_\tw$ obtained as the coefficient of $z^{-1}$ in $Y_\tw(\kappa_g,z)$, where
\begin{gather}\label{eqn:va:kappag}
	\kappa_g:=
	2b^+(-\tfrac12)b^-(-\tfrac12)\vv
	+
	\sum_{j=1}^{12} 2v_{g,j}^+(-\tfrac12)v_{g,j}^-(-\tfrac12)\vv,
\end{gather}
define $\imath^\tw$ to be the operator that is the identity on $M$, and multiplication by $i$ on $M_\tw$,
and let $\imath^{\tw,\odd(g)}$ be the operator that is $\imath^\tw$ if $o(g)$ is odd, 
and the identity otherwise.
We
now define $\R_g$ by setting
\begin{gather}\label{eqn:Rg}
	\R_g(\tau,z)
	:=
	-\lim_{u\to 1}
	\tr(
	(-1)^F
	\imath^{\tw,\odd(g)}
	gy^{J(0)}
	u^{K_g(0)}
	q^{L(0)-\frac{c}{24}}|\msn_\tw).
\end{gather}

The following theorem is the main result of this section.
\begin{thm}\label{thm:va:wjf}
For each $g\in M_{24}$ the function $\R_g$ is a well-defined weak Jacobi form
of weight $0$ and index $2$ for $\Gamma_0(o(g))$ that satisfies (\ref{eqn:va:phigtoPg}).
Moreover, the $\R_g$ are the bigraded traces associated to a bigraded module for $M_{24}$.
\end{thm}

Before proving Theorem \ref{thm:va:wjf} we present an explicit expression for $\R_g$, similar to (\ref{eqn:va:phi}).
For this we choose 
$\nu_{g,j}\in \CC$, 
for each $g\in M_{24}$ and $1\leq j\leq 12$, such that 
\begin{gather}\label{eqn:gvgjpm}
	g v_{g,j}^\pm = \nu_{g,j}^{\pm 2}v_{g,j}^\pm,
\end{gather}
where $v_{g,j}^\pm$ is as in (\ref{eqn:vgjmpvgkpm}), and we also require that
\begin{gather}\label{eqn:gprod}
	g\vv_{\tw,\gt{v},g} = \nu_{g}\vv_{\tw,\gt{v},g}
\end{gather}
in 
$\CM(\gt{v})$ (cf.\ (\ref{eqn:CMgtv=Cliffgtvetc})), 
where 
$\nu_g:=\prod_{j=1}^{12}\nu_{g,j}$.
Note that if we just require (\ref{eqn:gvgjpm}), then the left-hand side of (\ref{eqn:gprod}) is either the right-hand side of (\ref{eqn:gprod}) or its negative, 
so in practical terms, the problem of choosing $\nu_{g,j}$ as above is just a matter of choosing a square root of the eigenvalue of $g$ attached to the eigenvector $v_{g,j}^+$, for each $j$, and then replacing one of these with its negative in the case that (\ref{eqn:gprod}) doesn't already hold.

Checking the character table of $M_{24}$ (see Table \ref{tab:ct}) we observe that for every $g\in M_{24}$ the $g$-fixed subspace of the unique non-trivial $24$-dimensional representation is at least $2$-dimensional. 
Thus, after relabelling if necessary, we may assume 
that $\nu_{g,1}=1$. 

For notational convenience we recall the symbol $\varepsilon(d)$, typically defined just for $d$ odd, which appears in the theory of modular forms of half-integer weight: We have 
\begin{gather}
\varepsilon(d):=
\begin{cases}
1&\text{when $d\equiv 1 \xmod 4$,}\\
i&\text{when $d\equiv 3 \xmod 4$.}
\end{cases}
\end{gather}  

With $\nu_{g,j}$ chosen as above we now define
\begin{gather}
	\eta_{\pm g}(\tau):=q\prod_{n>0}\prod_{j=1}^{12}(1\mp\nu_{g,j}^{-2}q^n)(1\mp\nu_{g,j}^2q^n),
	\\
	C_{-g}:=(-i)\varepsilon((-1)^{o(g)})\prod_{j=1}^{12}(\nu_{g,j}+\nu_{g,j}^{-1}),
	\\
	D_{g}:=(-i)\varepsilon((-1)^{o(g)})\prod_{j=2}^{12}(\nu_{g,j}-\nu_{g,j}^{-1}).
\end{gather}
See Table \ref{tab:da} for the values of the $C_{-g}$ and $D_g$. Table \ref{tab:da} also specifies the $\eta_{\pm g}$ concretely: If $g$ has cycle shape $k_1^{m_1}\cdot k_2^{m_2}\dots$ then $\eta_g(\tau) = \eta(k_1\tau)^{m_1}\eta(k_2\tau)^{m_2}\dots$, and $\eta_{-g}$ is obtained from $\eta_g$ by replacing each $\eta(k\tau)$, for $k$ odd, with $\eta(2k\tau)\eta(k\tau)^{-1}$ (and leaving $\eta(k\tau)$, for $k$ even, as it is).

It develops that the product defining $C_{-g}$ vanishes whenever $o(g)$ is even,  so it is not wrong to simply write $C_{-g}=\prod_{j=1}^{12}(\nu_{g,j}+\nu_{g,j}^{-1})$. 
Note that we have $C_{-g}=2\prod_{j=2}^{12}(\nu_{g,j}+\nu_{g,j}^{-1})$, according to our assumption on $\nu_{g,1}$.
The next result follows directly from our construction.
\begin{pro}\label{pro:va:phig}
For $g\in M_{24}$ we have
\begin{gather}\label{eqn:va:phig}
\begin{split}
	\R_g(\tau,z) 
	&=
	-\frac12\frac{\theta_4(\tau,2z)}{\theta_4(\tau,0)}\frac{\eta_{g}(\frac{\tau}{2})}{\eta_{g}(\tau)}
	+\frac12\frac{\theta_3(\tau,2z)}{\theta_3(\tau,0)}\frac{\eta_{-g}(\frac{\tau}{2})}{\eta_{-g}(\tau)}
	\\
	&\quad
	-\frac12\frac{\theta_2(\tau,2z)}{\theta_2(\tau,0)}C_{-g}\eta_{-g}({\tau})
	-\frac12\frac{i\theta_1(\tau,2z)}{\eta(\tau)^3}D_{g}\eta_g(\tau).
\end{split}
\end{gather}
\end{pro}

\begin{proof}[Proof of Theorem \ref{thm:va:wjf}.]
We first note that the expression (\ref{eqn:Rg}) for $\R_g$ is well-defined because $\kappa_g$ (see (\ref{eqn:va:kappag})) is independent of any choices made in specifying the vectors $v_{g,j}^\pm$ of (\ref{eqn:vgjmpvgkpm}). Given this, the fact that $\R_g$ is a weak Jacobi form of weight $0$ and index $2$ for $\Gamma_0(o(g))$ (possibly with character), for each $g\in M_{24}$, can be checked directly using (\ref{eqn:va:phig}). The specialization identity (\ref{eqn:va:phigtoPg}) also follows directly from (\ref{eqn:va:phig}).
Finally, the $\R_g$ are the bigraded traces associated to a bigraded $M_{24}$-module by construction.
\end{proof}

\section{Concluding Remarks}

In this work we have interpreted the rational homotopy groups of a K3 surface as modules\footnote{The ``first'' rational homotopy group $\varrho_1({\rm K3})$ is a virtual module for $M_{24}$ in our interpretation.} 
$R_j$
for the sporadic simple Mathieu group $M_{24}$ (see Theorem \ref{thm:rh:rhoj}),
and we have 
constructed a vertex algebra that associates Jacobi forms $\R_g$ to the elements of 
$M_{24}$ in a compatible way (see Theorem \ref{thm:va:wjf}). 
This is reminiscent of Mathieu moonshine (see 
\cite{MR985505,MR2793423,MR2748168,MR2985326,MR3021323,MR3539377}),
wherein Jacobi forms that generalize the K3 elliptic genus are associated to the elements of $M_{24}$. 
Given the appearance of $M_{24}$ and K3 surfaces in both instances, it is natural to ask if these two phenomena are related.

An important difference between these two situations is that
in the setting of rational K3 homotopy the forms $\R_g$ arising (see (\ref{eqn:Rg})) are weak Jacobi forms of weight 0 and index 2 (with level), whereas in Mathieu moonshine they are weak Jacobi forms $Z_g$ of weight 0 and index 1 (with level).
Interestingly, there is a concrete connection between the $g=e$ cases of these two families: 
They are different specializations of the single weak Jacobi form
\begin{gather}
\label{eqn:H}
\begin{split}
\Hzw(\tau,z,w) :=& 
	-\frac12\frac{\theta_4(\tau,z-w)\theta_4(\tau,z+w)}{\theta_4(\tau,0)^2}\frac{\eta(\frac{\tau}{2})^{24}}{\eta(\tau)^{24}}\\
	&+\frac12\frac{\theta_3(\tau,z-w)\theta_3(\tau,z+w)}{\theta_3(\tau,0)^2}\frac{\eta({\tau})^{48}}{\eta(\frac{\tau}{2})^{24}\eta(2\tau)^{24}}\\
	&-\frac12\frac{\theta_2(\tau,z-w)\theta_2(\tau,z+w)}{\theta_2(\tau,0)^2}2^{12}\frac{\eta({2\tau})^{24}}{\eta(\tau)^{24}}.
\end{split}
\end{gather}
Indeed, we have $\Hzw(\tau,z,z) = \R(\tau,z)$ (cf.\ (\ref{eqn:va:phi})), and $\Hzw(\tau,z,0)=Z(\tau,z)$.

Note that $\Hzw$ in (\ref{eqn:H}) is a weak Jacobi form of {lattice index} (cf.\ e.g.\ \cite{MR3123592,MR3991423}) for the lattice $A_1\oplus A_1$, whereas a Jacobi form of lattice index $A_1$ is a Jacobi form of index $1$ in the sense of \cite{MR781735}.
Roughly, $\Hzw$ is like a Jacobi form in the classical sense, except that it has two independent elliptic variables with index $1$ in each of them.

The lattice-index Jacobi form $\Hzw$ appears in \cite{MR3668969} (denoted there by $Z^{s\natural}(\tau,z,w)$),  in connection with the enumerative geometry of K3 surfaces.
Specifically, $\Hzw$ is related to (reduced refined) K3 Gopakumar--Vafa invariants, and refined K3 BPS invariants. 
Let $\LL$ denote the Leech lattice, being the unique unimodular even lattice of rank $24$ with no vectors of square-length $2$.
In \S~3 of \cite{MR3668969} a twining $\Hzw_g$ of $\Hzw$ is defined, for each symmetry $g$ of the Conway group $\Co_0:=\Aut(\LL)$ that fixes a $4$-dimensional subspace of $\LL\otimes \CC$. 
The construction of the $\Hzw_g$ is closely related to what we have presented in \S~\ref{sec:va}: 
The $\Hzw_g$ are defined using the spaces that we have denoted $A(\gt{v})$ and $A(\gt{v})_\tw$.
Moreover, it can be checked that $\R_g(\tau,z) = \Hzw_g(\tau,z,z)$ for each $g\in M_{24}<\Co_0$ that fixes a $4$-dimensional subspace of $\LL\otimes \CC$.

This raises the question of whether or not there exist lattice-index Jacobi forms $\Hzw_g$, for each $g\in M_{24}$, such that $\Hzw_g(\tau,z,z)=\R_g(\tau,z)$ and $\Hzw_g(\tau,z,0)=Z_g(\tau,z)$. 
Unfortunately the answer is negative because it can be checked that $\Hzw_g(\tau,z,0)$ differs from $Z_g(\tau,z)$ for certain $g$, e.g.\ $g$ with cycle shape $3^8$ (see \cite{MR3465528}). 
However, it may be interesting to look more closely at the relationship between rational K3 homotopy, enumerative K3 geometry, and K3 compactification of string theory.
  
We now explain a geometric aspect to the coincidence $\Hzw(\tau,z,z)=\R(\tau,z)$. 
A key motivation for $\R(\tau,z)$ is the fact that its $q$-constant term (see \ref{eqn:va:limImtauphi}) is (a nomalization of) the K3 Poincar\'e series (see (\ref{eqn:PX})).
This fact can be read off from the K3 Hodge diamond (\ref{eqn:Hodgediamond}).
\begin{gather}\label{eqn:Hodgediamond}
\begin{array}{ccccccc}
& & & 1 & & & \\
& & 0 & & 0 & & \\
& 1 & & 20 & & 1 & \\
& & 0 & & 0 & & \\
& & & 1 & & &
\end{array}
\end{gather}
We can motivate $\Hzw$ by noting that its $q$-constant term is the (normalized) K3 Hodge polynomial,
\begin{gather}
	\Hzw(\tau,z,w) = y^{-1}v^{-1} + yv^{-1} + 20 + y^{-1}v + yv + O(q), 
\end{gather}
which fully encodes the data of (\ref{eqn:Hodgediamond}).
It would be interesting to have a homotopy theoretic interpretation for the K3 Hodge polynomial.

Finally we mention the problem of generalizing the results of this work to other manifolds.

\section*{Acknowledgements}

F.C.\ and Y.H.\ thank Vyacheslav Lysov for discussions. 
F.C.\ thanks Alessandro Mininno for comments on an early draft. 
F.C.\ and Y.H.\ are supported by the Leverhulme Research Project Grant 2022 ``Topology from cosmology: axions, astrophysics and machine learning". 
F.C.\ is supported by the Italian Ministry of Universities and Research (MUR) through the grant “StringGeom” (grant CUP code no.~I33c25000420006).
J.D.\ gratefully acknowledges support from Academia Sinica (AS-IA-113-M03) and the National Science and Technology Council of Taiwan (112-2115-M-001-006-MY3).

\appendix

\section{Plethysm}\label{sec:rtp}
	
Here we use a general representation-theoretic construction to ``upgrade'' the Plethystic expansion
(\ref{eqn:rh:POmegaVple}) from an identity of power series with integer coefficients to an identity of power series with (virtual) representations for coefficients. To explain this let 
Let $G$ be a finite group and let $R$ be the representation ring of (finitely generated modules for) $G$. 
Concretely, we may identify $R$ with the abelian of group of formal integer sums of equivalence classes of irreducible representations of $G$. 
This abelian group becomes a ring when we let the tensor product of $G$-modules define the multiplication.
(We refer to \S~3 of \cite{MR4712414} for a detailed description of $R$.)

In addition to tensor product, we also have the operations of alternating and symmetric powers on $R$.
Write $\Lambda^k(U)$ for the $k$-th alternating power of $U$, and write $S^k(U)$ for the $k$-th symmetric power of $U$, and recall that these operations satisfy
\begin{gather}\label{eqn:LambdakUSkU}
\Lambda^k(-U)=(-1)^kS^k(U)
\end{gather}
for arbitrary $U\in R$. (Cf.\ loc.\ cit.)

Now consider the ring $R[[x]]$, of formal power series with coefficients in $R$. 
Using alternating and symmetric powers we may define maps $R\to R[[x]]$ by setting
\begin{gather}\label{eqn:LambdaxUSxU}
	\Lambda_x(U) := \sum_{k\geq 0 }\Lambda^k(U)x^k,\quad
	S_x(U) := \sum_{k\geq 0 }S^k(U)x^k.
\end{gather}
Then $\Lambda_x(U+V) = \Lambda_x(U)\Lambda_x(V)$, and similarly for $S_x$, and as a consequence of (\ref{eqn:LambdakUSkU}) we have $\Lambda_{x}(-U) = S_{-x}(U)$. It follows that 
\begin{gather}
	\Lambda_{x}(U)S_{-x}(U) = 1
\end{gather}
in $R[[x]]$, for all $U\in R$.
In particular, $\Lambda_x(U)$ and $S_x(U)$ are invertible in $R[[x]]$. 

The full subgroup $R[[x]]^*$ of invertible elements of $R[[x]]$ is composed of the series $f\in R[[x]]$ with $f(0)=\pm 1$.
The series $\Lambda_x(U)$ and $S_x(U)$ both have this form, for all $U\in R$, but of course there are elements in $R[[x]]^*$ that are not of this form.
With the next two results we show that the constructions (\ref{eqn:LambdaxUSxU}) generate $R[[x]]^*$, in a certain sense.
\begin{lem}\label{lem:f=prodLambdaxnUn}
For any invertible $f\in R[[x]]$ with $f(0)=1$ there exist $U_n\in R$, for $n> 0$, such that 
\begin{gather}\label{eqn:f=prodLambdaxnUn}
	f(x) = \prod_{n>0} S_{x^n}(U_n).
\end{gather}
\end{lem}
Note that the infinite product on the right-hand side of (\ref{eqn:f=prodLambdaxnUn}) makes sense, because only finitely many factors involve any given positive power of $x$,
so only finitely many summands appear in the the coefficient of any given power of $x$ in (\ref{eqn:f=prodLambdaxnUn}).
\begin{proof}[Proof of Lemma \ref{lem:f=prodLambdaxnUn}]
Suppose that $f_n = 1 + g_n(x)x^n$ for some $g_n\in R[[x]]$, for some positive integer $n$. 
Then 
$f_n(x) = 1+g_n(0)x^n + O(x^{n+1})$, and $S_{x^n}(-g_n(0)) = 1-g_n(0)x^n +O(x^{n+1})$, 
so $f_n(x)S_{x^n}(-g_n(0)) = 1+g_{n+1}(x)x^{n+1}$ for some $g_{n+1}\in R[[x]]$. 
We obtain the lemma by applying this procedure iteratively, starting with $f_1 = f$, and taking $U_n = -g_n(0)$ at each iteration.
\end{proof}

We are ready to state and prove our representation-theoretic counterpart to the plethystic expansion (\ref{eqn:rh:POmegaVple}).
\begin{lem}\label{lem:repthyple}
For any invertible $f\in R[[x]]$ with $f(0)=1$ there exist $U_n\in R$, for $n\geq 0$, such that 
\begin{gather}\label{eqn:repthyple}
	f(x) = \prod_{j>0} \Lambda_{x^{2j-1}}(U_{2j-1}) S_{x^{2j}}(U_{2j}).
\end{gather}
\end{lem}
\begin{proof}
According to Lemma \ref{lem:f=prodLambdaxnUn} we have $f(-x)=\prod_{n>0}S_{x^n}(U_n')$ for some $U_n'\in R$.
Then $f(x) = \prod_{j>0}S_{-x^{2j-1}}(U_{2j-1}')S_{x^{2j}}(U_{2j}')$.
To obtain (\ref{eqn:repthyple}) we take $U_n = (-1)^nU_n'$.
\end{proof}

In Lemma \ref{lem:repthyple}, a factor $\Lambda_{x^n}(U_n)$ serves as the representation-theoretic counterpart to $(1+x^n)^{d}$, for $d =\dim U_n$, because if we write $\dim$ for the operator $R[[x]]\to \ZZ[[x]]$ that acts coefficient-wise as the usual dimension operator on $R$, then 
\begin{gather}
	\dim \Lambda_{x^n}(U_n) = (1+x^n)^{\dim U_n}.
\end{gather}
Similarly, 
\begin{gather}\label{eqn:dimSxnU}
	\dim S_{x^n}(U_n) = (1-x^n)^{-\dim U_n}
	.
\end{gather}

\section{Bounds}\label{sec:bnd}

In this section we establish the bounds that we require for the proof of Theorem \ref{thm:rh:rhoj}.

To begin we replace $x$ with $-x$ in (\ref{eqn:rh:def-rhojg}) in order to obtain
\begin{gather}\label{eqn:rh:rhoj-1}
\log(1+\rjg_1(g) x+ x^2) 
=-\sum_{j,k>0}
(-1)^{j}
\frac1k
\rjg_j(g^k) x^{jk}
	=
	-\sum_{m>0}\sum_{d|m}
	(-1)^{d}
	\frac{d}{m}\rjg_d(g^{\frac{m}{d}})x^m
.
\end{gather}
Now set $\alpha(g):=\frac12\rjg_1(g)+\frac12\sqrt{\rjg_1(g)^2-4}$. Then we have
\begin{gather}\label{eqn:factorP_g}
1+\rjg_1(g)x+x^2=(1+\alpha(g) x)(1+\alpha(g)^{-1}x).
\end{gather}
After plugging (\ref{eqn:factorP_g}) into the left-hand side of (\ref{eqn:rh:rhoj-1}) and extracting coefficients we obtain 
\begin{gather}\label{eqn:rh:rhoj-2}
	(-1)^{m}
	a(g,m)
	=
	\sum_{d|m}
	(-1)^{d}
	d
	\rjg_d(g^{\frac{m}{d}})
	,
\end{gather}
where, to ease notation, we have used $a(g,m):=\alpha(g)^m+\alpha(g)^{-m}$. 

We would like to use (\ref{eqn:rh:rhoj-2}) to rewrite $\rjg_m(g)$ in terms of $a(g,m)$. 
We can actually do this directly in the special case that $m$ is coprime to the order of $g$, for in that case (\ref{eqn:rh:rhoj-2}) reduces to
\begin{gather}\label{eqn:rh:agm-cop}
	(-1)^{m}
	a(g,m)
	=
	\sum_{d|m}
	(-1)^d
	d\rjg_d(g)
	.
\end{gather}
M\"obius inversion then gives us
\begin{gather}\label{eqn:rh:rhomg-cop}
	\rjg_m(g) = 
	(-1)^m
	\frac1m
	\sum_{d|m}\mu\left(\frac{m}{d}\right)
	(-1)^da(g,d).
\end{gather}
In the next two lemmas we establish counterparts to (\ref{eqn:rh:rhomg-cop}) for general $m$. For this we agree to define $a(g,r):=0$ in case $r$ is not an integer.
\begin{lem}\label{lem:rhomg-ogp}
Suppose that $g\in M_{24}$ has prime power order $o(g)=p^k$, and let $m$ be a positive integer. Then we have
\begin{gather}\label{eqn:rhomg-ogp}
	\rjg_m(g)=(-1)^m\frac{1}{m}\sum_{d|n}\mu\left(\frac{n}{d}\right)
	\left(
	(-1)^{ds}a(g,ds) - (-1)^{\frac{ds}{p}}a\left(g^p,\frac{ds}{p}\right)
	\right)
\end{gather}
where $n$ is the largest divisor of $m$ that is coprime to $o(g)$, and $s:=\frac{m}{n}$.
\end{lem}
\begin{proof}
Let $k>0$ and suppose that $o(g)=p^k$.
Let $m>0$ and write 
$m=np^v$, where $v\geq 0$ and $(n,p)=1$. 
If $v=0$ then we require to check that 
(\ref{eqn:rhomg-ogp}) agrees with (\ref{eqn:rh:rhomg-cop}), which it does
according to our convention that $a(g,r)=0$ in case $r\notin\ZZ$. So assume that $v>0$.
Then we may write (\ref{eqn:rh:rhoj-2}) in the form 
\begin{gather}\label{eqn:agm=sum_1}
	(-1)^{np^v}
	a(g,np^v)
	=
	\sum_{d|n}
	(-1)^{dp^v}
	dp^v
	\rjg_{dp^v}(g)
+	\sum_{d|np^{v-1}}
	(-1)^{d}
	d
	\rjg_{d}(g^{\frac{np^v}{d}}).
\end{gather}
Observe that the second sum on the right-hand side of (\ref{eqn:agm=sum_1}) is $(-1)^{np^{v-1}}a(g^p,np^{v-1})$ 
according to (\ref{eqn:rh:rhoj-2}). Thus we have
\begin{gather}\label{eqn:agm=sum_2}
	(-1)^{np^v}
	a(g,np^v)
	-
	(-1)^{np^{v-1}}a(g^p,np^{v-1})
	=
	\sum_{d|n}
	(-1)^{dp^v}
	dp^v
	\rjg_{dp^v}(g).
\end{gather}
Applying M\"obius inversion to (\ref{eqn:agm=sum_2}) we obtain
\begin{gather}\label{eqn:ogp2}
	\rjg_{np^v}(g)=(-1)^{np^v}\frac{1}{np^v}\sum_{d|n}\mu\left(\frac{n}{d}\right)
	\left(
	(-1)^{dp^v}a(g,dp^v)
	-
	(-1)^{dp^{v-1}}a(g^p,dp^{v-1})
	\right),
\end{gather}
which is just what we required to show.
\end{proof}
\begin{lem}\label{lem:rhomg-ogp1p2}
Suppose that $g\in M_{24}$ has order divisible by just two primes $p_1$ and $p_2$, and let $m$ be a positive integer. Then we have
\begin{gather}
\begin{split}\label{eqn:rhomg-ogp1p2}
	\rjg_m(g)=(-1)^m\frac{1}{m}\sum_{d|n}
	&\;\mu\left(\frac{n}{d}\right)
	\Big(
	(-1)^{ds}a(g,ds) - (-1)^{dsp_1^{-1}}a(g^{p_1},dsp_1^{-1})\\
	&- (-1)^{dsp_2^{-1}}a(g^{p_2},dsp_2^{-1})
	+ (-1)^{dsp_1^{-1}p_2^{-1}}a(g^{p_1p_2},dsp_1^{-1}p_2^{-1})
	\Big)
\end{split}
\end{gather}
where $n$ is the largest divisor of $m$ that is coprime to $o(g)$, and $s:=mn^{-1}$.
\end{lem}
\begin{proof}
The basic idea of the proof is the same as for Lemma \ref{lem:rhomg-ogp}. 
Namely, we apply (\ref{eqn:rh:rhoj-2}) to itself, and then apply M\"obius inversion, and use (\ref{eqn:rh:rhomg-cop}) to handle the edge cases where this procedure breaks down.

To implement this let $k_1,k_2>0$ and suppose that $o(g)=p_1^{k_1}p_2^{k_2}$.
Let $m>0$ and write $m=np_1^{v_1}p_2^{v_2}$ where $(n,p_1p_2)=1$.
If $v_1=v_2=0$ then, similar to the Proof of Lemma \ref{lem:rhomg-ogp}, the desired expression (\ref{eqn:rhomg-ogp1p2}) reduces to (\ref{eqn:rh:rhomg-cop}).

If $v_1>0$ and $v_2=0$ then (\ref{eqn:rhomg-ogp1p2}) reduces to the statement that
\begin{gather}
\label{eqn:rhomg-ogp1p2_v20}
	\rjg_{np_1^{v_1}}(g)=(-1)^{np_1^{v_1}}\frac{1}{np_1^{v_1}}\sum_{d|n}\mu\left(\frac{n}{d}\right)\Big(
	(-1)^{dp_1^{v_1}}a(g,dp_1^{v_1}) - (-1)^{dp_1^{v_1-1}}a(g^{p_1},dp_1^{v_1-1})
	\Big).
\end{gather}
To show this we 
rewrite (\ref{eqn:rh:rhoj-2}) in the form
\begin{gather}\label{rjg-ogp1p2_v20_sum2}
	(-1)^{np_1^{v_1}}a(g,np_1^{v_1})
	=\sum_{d|n}
	(-1)^{dp_1^{v_1}}dp_1^{v_1}
	\rjg_{dp_1^{v_1}}(g)
	+\sum_{d|np_1^{v_1-1}}
	(-1)^{d}d\rjg_{d}(g^{\frac{np_1^{v_1}}{d}}),
\end{gather}
and apply (\ref{eqn:rh:rhoj-2}) again in order to identify the second term on the right-hand side of 
(\ref{rjg-ogp1p2_v20_sum2}) with $(-1)^{np_1^{v_1-1}}a(g^{p_1},np_1^{v_1-1})$. This allows us to rewrite (\ref{rjg-ogp1p2_v20_sum2}) as
\begin{gather}\label{rjg-ogp1p2_v20_sum2_1}
	(-1)^{np_1^{v_1}}a(g,np_1^{v_1}) - 
	(-1)^{np_1^{v_1-1}}a(g^{p_1},np_1^{v_1-1})
	=\sum_{d|n}
	(-1)^{dp_1^{v_1}}dp_1^{v_1}
	\rjg_{dp_1^{v_1}}(g).
\end{gather}
We obtain (\ref{eqn:rhomg-ogp1p2_v20}) from (\ref{rjg-ogp1p2_v20_sum2_1}) by M\"obius inversion.

The case that $v_1=0$ and $v_2>0$ is the same so assume now that $v_1,v_2>0$. 
Then we may write (\ref{eqn:rh:rhoj-2}) in the form
\begin{gather}
\begin{split}\label{eqn:rh:lem2_a}
	(-1)^{np_1^{v_1}p_2^{v_2}}a(g,np_1^{v_1}p_2^{v_2})
	=&\;\sum_{d|n}
	(-1)^{dp_1^{v_1}p_2^{v_2}}dp_1^{v_1}p_2^{v_2}
	\rjg_{dp_1^{v_1}p_2^{v_2}}(g)
	\\&
	+\sum_{d|np_1^{v_1-1}}
	(-1)^{dp_2^{v_2}}dp_2^{v_2}\rjg_{dp_2^{v_2}}(g^{\frac{np_1^{v_1}}{d}})
	\\
	&
	+\sum_{d|np_2^{v_2-1}}
	(-1)^{dp_1^{v_1}}dp_1^{v_1}\rjg_{dp_1^{v_1}}(g^{\frac{np_2^{v_2}}{d}})
	\\&
	+\sum_{d|np_1^{v_1-1}p_2^{v_2-1}}(-1)^dd\rjg_d(g^{\frac{np_1^{v_1}p_2^{v_2}}{d}}).
\end{split}
\end{gather}
For the second sum on the right-hand side of (\ref{eqn:rh:lem2_a}) we note that
\begin{gather}
\Div(np_1^{v_1-1}p_2^{v_2}) = \Div(np_1^{v_1-1})p_2^{v_2} \sqcup \Div(np_1^{v_1-1}p_2^{v_2-1}).
\end{gather}
Applying this to (\ref{eqn:rh:rhoj-2}) we get
\begin{gather}\label{eqn:rh:lem2_a2}
\begin{split}
	(-1)^{np_1^{v_1-1}p_2^{v_2}}a(g^{p_1},np_1^{v_1-1}p_2^{v_2})
	=&
	\sum_{d|np_1^{v_1-1}}(-1)^{dp_2^{v_2}}dp_2^{v_2}\rjg_{dp_2^{v_2}}(g^{\frac{np_1^{v_1}}{d}})
	\\
	&+\sum_{d|np_1^{v_1-1}p_2^{v_2-1}}(-1)^dd\rjg_d(g^{\frac{np_1^{v_1}p_2^{v_2}}{d}}).
\end{split}
\end{gather}
Symmetrically we have
\begin{gather}\label{eqn:rh:lem2_a3}
\begin{split}
	(-1)^{np_1^{v_1}p_2^{v_2-1}}a(g^{p_2},np_1^{v_1}p_2^{v_2-1})
	=&
	\sum_{d|np_2^{v_2-1}}(-1)^{dp_1^{v_1}}dp_1^{v_1}\rjg_{dp_1^{v_1}}(g^{\frac{np_2^{v_2}}{d}})
	\\
	&+\sum_{d|np_1^{v_1-1}p_2^{v_2-1}}(-1)^dd\rjg_d(g^{\frac{np_1^{v_1}p_2^{v_2}}{d}}),
\end{split}
\end{gather}
while for the last sum in each of (\ref{eqn:rh:lem2_a}) and (\ref{eqn:rh:lem2_a2}-\ref{eqn:rh:lem2_a3}) we have
\begin{gather}\label{eqn:rh:lem2_a4}
	(-1)^{np_1^{v_1-1}p_2^{v_2-1}}a(g^{p_1p_2},np_1^{v_1-1}p_2^{v_2-1})
	=\sum_{d|np_1^{v_1-1}p_2^{v_2-1}}(-1)^dd\rjg_d(g^{\frac{np_1^{v_1}p_2^{v_2}}{d}}).
\end{gather}

Substituting (\ref{eqn:rh:lem2_a2}-\ref{eqn:rh:lem2_a4}) into (\ref{eqn:rh:lem2_a}) we obtain
\begin{gather}
\begin{split}\label{eqn:rh:lem2_aa}
	&(-1)^{np_1^{v_1}p_2^{v_2}}a(g,np_1^{v_1}p_2^{v_2})
	-(-1)^{np_1^{v_1-1}p_2^{v_2}}a(g^{p_1},np_1^{v_1-1}p_2^{v_2})\\
	&-(-1)^{np_1^{v_1}p_2^{v_2-1}}a(g^{p_2},np_1^{v_1}p_2^{v_2-1})
	+(-1)^{np_1^{v_1-1}p_2^{v_2-1}}a(g^{p_1p_2},np_1^{v_1-1}p_2^{v_2-1})\\
	&=\;\sum_{d|n}
	(-1)^{dp_1^{v_1}p_2^{v_2}}dp_1^{v_1}p_2^{v_2}
	\rjg_{dp_1^{v_1}p_2^{v_2}}(g),
\end{split}
\end{gather}
and the desired result (\ref{eqn:rhomg-ogp1p2}) now follows by M\"obius inversion.
\end{proof}
\begin{rmk}
It is clear that Lemmas \ref{lem:rhomg-ogp} and \ref{lem:rhomg-ogp1p2} are special cases of a more general result, that handles $g$ with order divisible by more than 2 primes. 
We do not need such a result here since there is no such element in $M_{24}$, but you could imagine carrying out a similar analysis with a more general group. 
With this in mind we formulate the identity 
\begin{gather}\label{eqn:rhomg-gen}
	\rjg_m(g)=(-1)^m\frac{1}{m}\sum_{d|n}\mu\left(\frac{n}{d}\right)\left(
	\sum_{P\subset \Pi(s)}(-1)^{\# P+ds\pi(P)^{-1}}
	a(g^{\pi(P)},ds\pi(P)^{-1})
	\right)
\end{gather}
which is applicable to general orders. In (\ref{eqn:rhomg-gen}) we write $\Pi(n)$ for the set of prime divisors of an integer $n$, and, given a set $S$ of primes, write $\pi(S)$ as a short hand for the product $\prod_{p\in S}p$. We also take $n$ to be the largest divisor of $m$ that is coprime to $o(g)$, and $s:=mn^{-1}$, as before.
\end{rmk}

Now that we have concrete expressions for $\rjg_m(g)$ in terms of the $a(g,m)$ we wish to use these to derive bounds on the $\rjg_m(g)$. For this we use the following.
\begin{lem}\label{lem:rh:boundagm}
Let $g\in M_{24}$ and let $m$ be a positive integer.
\begin{itemize}
\item
If $g$ is of class $1A$ then $22^m(\frac{481^m}{482^m}+\frac{1}{483^m})<a(g,m)< 22^m(\frac{482^m}{483^m}+\frac{1}{482^m})$.
\item
If $g$ is of class $2A$ then $6^m(\frac{33^m}{34^m}+\frac{1}{35^m})<a(g,m)< 6^m(\frac{34^m}{35^m}+\frac{1}{34^m})$.
\item
If $g$ is of class $3A$ then $4^m(\frac{13^m}{14^m}+\frac{1}{15^m})<a(g,m)< 4^m(\frac{14^m}{15^m}+\frac{1}{14^m})$.
\item
If $g$ is not of class $1A$, $2A$ or $3A$ then $|a(g,m)|\leq 2$.
\end{itemize}
\end{lem}
\begin{proof}
By direct calculation we find that
\begin{gather}\label{eqn:boundagm-1}
	x\frac{x^2-3}{x^2-2}<\frac{x}{2}+\frac{\sqrt{x^2-4}}{2}  < x\frac{x^2-2}{x^2-1}
\end{gather}
for $x>2$. Taking $x=\rjg_1(g)$ in (\ref{eqn:boundagm-1}) for $g$ of class $1A$, $2A$ and $3A$ we obtain
\begin{gather}\label{eqn:boundagm-2}
	22\frac{481}{482}<\alpha(1A)<22\frac{482}{483},\quad
	6\frac{33}{34}<\alpha(2A)<6\frac{34}{35},\quad
	4\frac{13}{14}<\alpha(3A)<4\frac{14}{15}.	
\end{gather}
By similar methods, or by applying $\alpha(g)^{-1}= \rjg_1(g)-\alpha(g)$ to (\ref{eqn:boundagm-2}), we obtain
\begin{gather}\label{eqn:boundagm-3}
	\frac{22}{483}<\alpha(1A)^{-1}<\frac{22}{482},\quad
	\frac{6}{35}<\alpha(2A)^{-1}<\frac{6}{34},\quad
	\frac{4}{15}<\alpha(3A)^{-1}<\frac{4}{14}.	
\end{gather}
The first three claims follow from (\ref{eqn:boundagm-2}-\ref{eqn:boundagm-3}). 
For the final claim we note that
the $g\in M_{24}$ that are not of class $1A$, $2A$ or $3A$ are exactly those for which we have $\rjg_1(g)\in\{0,\pm 1,\pm 2\}$.  
These, in turn, are exactly the cases that $\alpha(g)$ is a root of unity.
It follows that the absolute value of $a(g,m)=\alpha(g)^m+\alpha(g)^{-m}$ is at most $2$.
\end{proof}

\begin{lem}\label{lem:rh:rhom1A-lb}
For $m>3$ we have
\begin{gather}\label{eqn:rh:rhom1A-lb}
	\rjg_m(e) > \frac1m 22^m\left(\frac{481^m}{482^m}+\frac{1}{483^m}\right) - 22^{\frac m 2}.
\end{gather}
\end{lem}
\begin{proof}
Using (\ref{eqn:rh:rhomg-cop}) with $g=e$ we write
\begin{gather}\label{eqn:rh:rhom1A-lb-1}
	\rjg_m(e) = \frac1ma(e,m) + (-1)^m\frac1m\sum_{\substack{d|m\\d<m}}\mu\left(\frac{m}d\right)(-1)^da(e,d),
\end{gather}
the idea being that the first term on the right-hand side of (\ref{eqn:rh:rhom1A-lb-1}) dominates the growth of $\rjg_m(e)$. 
Observe 
that any upper bound for $|a(e,\frac{m}{2})|$ bounds every summand within the summation on the right-hand side of (\ref{eqn:rh:rhom1A-lb-1}), and there are not more than $2m^{\frac12}$ terms in that summation, because there are not more than $2m^{\frac12}$ divisors of $m$. Using Lemma \ref{lem:rh:boundagm} to bound $a(e,m)$ from below, and $a(e,\frac{m}2)$ from above, we thus obtain
\begin{gather}
	\rjg_m(e) > \frac1m 22^m\left(\frac{481^m}{482^m}+\frac{1}{483^m}\right) - 2m^{\frac12}\frac1m 22^{\frac m2}\left(\frac{482^{\frac m2}}{483^{\frac m2}}+\frac{1}{482^{\frac m2}}\right).
\end{gather}
We obtain the simpler expression (\ref{eqn:rh:rhom1A-lb}) by observing that both $2m^{-\frac12}$ and $\frac{482^{\frac m2}}{483^{\frac m2}}+\frac{1}{482^{\frac m2}}$ are bounded above by $1$ for $m\geq 4$.
\end{proof}

\begin{lem}\label{lem:rh:rhomg-ub}
For $g\in M_{24}$ with $g\neq e$ and for $m>11$ we have
\begin{gather}\label{eqn:rh:rhomg-ub}
	|\rjg_m(g)| < 
	\frac1m 6^m 
	+ \frac1m 22^{\frac{m}2} 
	+6^{\frac{m}3} + 22^{\frac{m}{6}}.
\end{gather}
\end{lem}
\begin{proof}
We prove this conjugacy class by conjugacy class, applying Lemma \ref{lem:rhomg-ogp} in the case that $g$ has prime-power order, and applying Lemma \ref{lem:rhomg-ogp1p2} otherwise.
Since the arguments are very similar in each case, we treat $[g]={\rm 2A}$ and $[g]={\rm 6A}$ in detail here, and leave the remaining cases as exercises for the reader.

In the case that $[g]={\rm 2A}$ we begin by applying (\ref{eqn:rhomg-ogp}) with $g$ of order $2$. 
In that notation we have $s=2^v$ and $m=2^vn$, for some positive $v$, if $m$ is even, and $s=1$ and $n=m$ if $m$ is odd. We obtain
\begin{gather}\label{eqn:rh:rhomg-ub-1}
\begin{split}
\rjg_m(g) =& \frac{1}{m}a(g,m) - (-1)^{\frac m2}\frac 1ma\left(e,\frac{m}2\right)\\
&+(-1)^m\frac1m\sum_{\substack{d|n\\d<n}}\mu\left(\frac{n}{d}\right)
	\left(
	(-1)^{ds}a(g,ds) - (-1)^{\frac{ds}{2}}a\left(e,\frac{ds}{2}\right)
	\right),
\end{split}
\end{gather}
where $a(g,\frac m2)$ and $a(e,\frac{ds}{2})$ are omitted (i.e.\ regarded as zero) in the case that $m$ is odd.
Observe now that an upper bound for $|a(g,\frac{m}{3})|$ is an upper bound for every term $a(g,ds)$ in (\ref{eqn:rh:rhomg-ub-1}), since the smallest non-trivial divisor of $n$ is at least $3$ since $n$ is odd.
Observe also that an upper bound for $|a(e,\frac{m}{6})|$ is an upper bound for every term $a(e,\frac{ds}{2})$ in (\ref{eqn:rh:rhomg-ub-1}) for the same reason.
Applying Lemma \ref{lem:rh:boundagm}, and using the fact that
there are no more than $2{m}^{\frac12}$ divisors of $m$ we obtain 
\begin{gather}\label{eqn:rh:rhomg-ub-2}
\begin{split}
	|\rjg_m(g)| < &
	\frac1m 6^m \left(\frac{34^m}{35^m}+\frac{1}{34^m}\right) 
	+ \frac1m 22^{\frac{m}2} \left(\frac{482^{\frac{m}{2}}}{483^{\frac{m}{2}}}+\frac{1}{482^{\frac{m}{2}}}\right)\\
	&+2m^{-\frac12}6^{\frac{m}3}\left(\frac{34^{\frac m3}}{35^{\frac m3}}+\frac{1}{34^{\frac m3}}\right)  + 2m^{-\frac12}22^{\frac{m}{6}}\left(\frac{482^{\frac{m}{6}}}{483^{\frac{m}{6}}}+\frac{1}{482^{\frac{m}{6}}}\right),
\end{split}
\end{gather}
and we get the desired expression (\ref{eqn:rh:rhomg-ub}) by observing that $2m^{-\frac12}$ and 
the parenthetical terms in the right-hand side of (\ref{eqn:rh:rhomg-ub-2}) 
are all bounded above by $1$ for $m\geq 12$.

Now suppose that $[g]={\rm 6A}$. 
Let $m>0$ and write $m=ns$ with $(n,6)=1$. Then (\ref{eqn:rhomg-ogp1p2}) specializes to
\begin{gather}
\begin{split}\label{eqn:lemupo6}
	\rjg_m(g)=&\;
	\frac1{m}
	a(g,m)
	-	
	(-1)^{\frac{m}{2}}
	\frac1{m}
	a\left(g^2,\frac{m}{2}\right)
	-
	\frac1{m}
	a\left(g^3,\frac{m}{3}\right)
	+	
	(-1)^{\frac{m}{6}}
	\frac1{m}
	a\left(e,\frac{m}{6}\right)
	\\
	&+(-1)^m\frac{1}{m}\sum_{\substack{d|n\\d<n}}
	\mu\left(\frac{n}{d}\right)
	\left(
	(-1)^{ds}a(g,ds) - (-1)^{\frac{ds}{2}}a\left(g^{2},\frac{ds}{2}\right)\right.\\
	&- \left.(-1)^{\frac{ds}{3}}a\left(g^{3},\frac{ds}{3}\right)
	+ (-1)^{\frac{ds}{6}}a\left(e,\frac{ds}{6}\right)
	\right).
\end{split}
\end{gather}
In this case a smallest non-trivial divisor of $n$ is $5$, so 
$|a(g^2,\frac{m}{10})|$ bounds every $a(g^2,\frac{ds}{2})$ in (\ref{eqn:lemupo6}), and $|a(g^3,\frac{m}{15})|$ bounds every $a(g^3,\frac{ds}{3})$,
and $|a(e,\frac{m}{30})|$ bounds every $a(e,\frac{ds}{6})$.
Note that $g^2\in [3A]$ and $g^3\in [2A]$.
Again applying Lemma \ref{lem:rh:boundagm}, and again using the fact that there are no more than $2m^{\frac12}$ divisors of $m$ we obtain 
\begin{gather}
\begin{split}
|\rjg_m(g)|<&\;4m^{-\frac12}
+\frac1m 4^{\frac{m}{2}} \left(\frac{14^{\frac{m}{2}}}{15^{\frac{m}{2}}}+\frac{1}{14^{\frac{m}{2}}}\right)
+\frac1m 6^{\frac{m}{3}} \left(\frac{34^{\frac m3}}{35^{\frac{m}{3}}}+\frac{1}{34^{\frac{m}{3}}}\right) 
\\
&+\frac1m 22^{\frac{m}6} \left(\frac{482^{\frac{m}{6}}}{483^{\frac{m}{6}}}+\frac{1}{482^{\frac{m}{6}}}\right)\\
&
+2m^{-\frac12} 4^{\frac{m}{10}} \left(\frac{14^{\frac{m}{10}}}{15^{\frac{m}{10}}}+\frac{1}{14^{\frac{m}{10}}}\right)
+2m^{-\frac12} 6^{\frac{m}{15}} \left(\frac{34^{\frac m{15}}}{35^{\frac{m}{15}}}+\frac{1}{34^{\frac{m}{15}}}\right) 
\\
&+2m^{-\frac12} 22^{\frac{m}{30}} \left(\frac{482^{\frac{m}{30}}}{483^{\frac{m}{30}}}+\frac{1}{482^{\frac{m}{30}}}\right).
\end{split}
\end{gather}
As before we may simplify this to 
\begin{gather}\label{eqn:rjgmgo6last}
|\rjg_m(g)|<2+\frac1m 4^{\frac{m}{2}}+\frac1m 6^{\frac{m}{3}}+\frac1m 22^{\frac{m}6}+4^{\frac{m}{10}} +6^{\frac{m}{15}} + 22^{\frac{m}{30}}.
\end{gather}
It is now straightforward to check that the right-hand side of (\ref{eqn:rjgmgo6last}) is bounded above by the right-hand side of (\ref{eqn:rh:rhomg-ub}) for any positive integer $m$.
\end{proof}

\begin{sidewaystable}

\section{Characters}\label{sec:ct}

\begin{center}
\begin{footnotesize}
\caption{Character table of $M_{24}$}\label{tab:ct}
\smallskip
\begin{tabular}{c@{ }|c|@{ }r@{ }r@{ }r@{ }r@{ }r@{ }r@{ }r@{ }r@{ }r@{ }r@{ }r@{ }r@{ }r@{ }r@{ }r@{ }r@{ }r@{ }r@{ }r@{ }r@{ }r@{ }r@{ }r@{ }r@{ }r@{ }r} \toprule
$[g]$	&{FS}&1A	&2A	&2B	&3A	&3B	&4A	&4B	&4C	&5A	&6A	&6B	&7A	&7B	&8A	&10A	&11A	&12A	&12B	&14A	&14B	&15A	&15B	&21A	&21B	&23A	&23B	\\
	\midrule
$[g^{2}]$	&&1A	&1A	&1A	&3A	&3B	&2A	&2A	&2B	&5A	&3A	&3B	&7A	&7B	&4B	&5A	&11A	&6A	&6B	&7A	&7B	&15A	&15B	&21A	&21B	&23A	&23B	\\
$[g^{3}]$	&&1A	&2A	&2B	&1A	&1A	&4A	&4B	&4C	&5A	&2A	&2B	&7B	&7A	&8A	&10A	&11A	&4A	&4C	&14B	&14A	&5A	&5A	&7B	&7A	&23A	&23B	\\
$[g^{5}]$	&&1A	&2A	&2B	&3A	&3B	&4A	&4B	&4C	&1A	&6A	&6B	&7B	&7A	&8A	&2B	&11A	&12A	&12B	&14B	&14A	&3A	&3A	&21B	&21A	&23B	&23A	\\
$[g^{7}]$	&&1A	&2A	&2B	&3A	&3B	&4A	&4B	&4C	&5A	&6A	&6B	&1A	&1A	&8A	&10A	&11A	&12A	&12B	&2A	&2A	&15B	&15A	&3B	&3B	&23B	&23A	\\
$[g^{11}]$	&&1A	&2A	&2B	&3A	&3B	&4A	&4B	&4C	&5A	&6A	&6B	&7A	&7B	&8A	&10A	&1A	&12A	&12B	&14A	&14B	&15B	&15A	&21A	&21B	&23B	&23A	\\
$[g^{23}]$	&&1A	&2A	&2B	&3A	&3B	&4A	&4B	&4C	&5A	&6A	&6B	&7A	&7B	&8A	&10A	&11A	&12A	&12B	&14A	&14B	&15A	&15B	&21A	&21B	&1A	&1A	\\
	\midrule
$\chi_{1}$	&$+$&$1$	&$1$	&$1$	&$1$	&$1$	&$1$	&$1$	&$1$	&$1$	&$1$	&$1$	&$1$	&$1$	&$1$	&$1$	&$1$	&$1$	&$1$	&$1$	&$1$	&$1$	&$1$	&$1$	&$1$	&$1$	&$1$	\\
$\chi_{2}$	&$+$&$23$	&$7$	&$-1$	&$5$	&$-1$	&$-1$	&$3$	&$-1$	&$3$	&$1$	&$-1$	&$2$	&$2$	&$1$	&$-1$	&$1$	&$-1$	&$-1$	&$0$	&$0$	&$0$	&$0$	&$-1$	&$-1$	&$0$	&$0$	\\
$\chi_{3}$	&$\circ$&$45$	&$-3$	&$5$	&$0$	&$3$	&$-3$	&$1$	&$1$	&$0$	&$0$	&$-1$	&$b_7$	&$\overline{b_7}$	&$-1$	&$0$	&$1$	&$0$	&$1$	&$-b_7$	&$-\overline{b_7}$	&$0$	&$0$	&$b_7$	&$\overline{b_7}$	&$-1$	&$-1$	\\
$\chi_{4}$	&$\circ$&$45$	&$-3$	&$5$	&$0$	&$3$	&$-3$	&$1$	&$1$	&$0$	&$0$	&$-1$	&$\overline{b_7}$	&$b_7$	&$-1$	&$0$	&$1$	&$0$	&$1$	&$-\overline{b_7}$	&$-b_7$	&$0$	&$0$	&$\overline{b_7}$	&$b_7$	&$-1$	&$-1$	\\
$\chi_{5}$	&$\circ$&$231$	&$7$	&$-9$	&$-3$	&$0$	&$-1$	&$-1$	&$3$	&$1$	&$1$	&$0$	&$0$	&$0$	&$-1$	&$1$	&$0$	&$-1$	&$0$	&$0$	&$0$	&$b_{15}$	&$\overline{b_{15}}$	&$0$	&$0$	&$1$	&$1$	\\
$\chi_{6}$	&$\circ$&$231$	&$7$	&$-9$	&$-3$	&$0$	&$-1$	&$-1$	&$3$	&$1$	&$1$	&$0$	&$0$	&$0$	&$-1$	&$1$	&$0$	&$-1$	&$0$	&$0$	&$0$	&$\overline{b_{15}}$	&$b_{15}$	&$0$	&$0$	&$1$	&$1$	\\
$\chi_{7}$	&$+$&$252$	&$28$	&$12$	&$9$	&$0$	&$4$	&$4$	&$0$	&$2$	&$1$	&$0$	&$0$	&$0$	&$0$	&$2$	&$-1$	&$1$	&$0$	&$0$	&$0$	&$-1$	&$-1$	&$0$	&$0$	&$-1$	&$-1$	\\
$\chi_{8}$	&$+$&$253$	&$13$	&$-11$	&$10$	&$1$	&$-3$	&$1$	&$1$	&$3$	&$-2$	&$1$	&$1$	&$1$	&$-1$	&$-1$	&$0$	&$0$	&$1$	&$-1$	&$-1$	&$0$	&$0$	&$1$	&$1$	&$0$	&$0$	\\
$\chi_{9}$	&$+$&$483$	&$35$	&$3$	&$6$	&$0$	&$3$	&$3$	&$3$	&$-2$	&$2$	&$0$	&$0$	&$0$	&$-1$	&$-2$	&$-1$	&$0$	&$0$	&$0$	&$0$	&$1$	&$1$	&$0$	&$0$	&$0$	&$0$	\\
$\chi_{10}$&$\circ$&$770$	&$-14$	&$10$	&$5$	&$-7$	&$2$	&$-2$	&$-2$	&$0$	&$1$	&$1$	&$0$	&$0$	&$0$	&$0$	&$0$	&$-1$	&$1$	&$0$	&$0$	&$0$	&$0$	&$0$	&$0$	&$b_{23}$	&$\overline{b_{23}}$	\\
$\chi_{11}$&$\circ$&$770$	&$-14$	&$10$	&$5$	&$-7$	&$2$	&$-2$	&$-2$	&$0$	&$1$	&$1$	&$0$	&$0$	&$0$	&$0$	&$0$	&$-1$	&$1$	&$0$	&$0$	&$0$	&$0$	&$0$	&$0$	&$\overline{b_{23}}$	&$b_{23}$	\\
$\chi_{12}$&$\circ$&$990$	&$-18$	&$-10$	&$0$	&$3$	&$6$	&$2$	&$-2$	&$0$	&$0$	&$-1$	&$b_7$	&$\overline{b_7}$	&$0$	&$0$	&$0$	&$0$	&$1$	&$b_7$	&$\overline{b_7}$	&$0$	&$0$	&$b_7$	&$\overline{b_7}$	&$1$	&$1$	\\
$\chi_{13}$&$\circ$&$990$	&$-18$	&$-10$	&$0$	&$3$	&$6$	&$2$	&$-2$	&$0$	&$0$	&$-1$	&$\overline{b_7}$	&$b_7$	&$0$	&$0$	&$0$	&$0$	&$1$	&$\overline{b_7}$	&$b_7$	&$0$	&$0$	&$\overline{b_7} $	&$b_7$	&$1$	&$1$	\\
$\chi_{14}$&$+$&$1035$	&$27$	&$35$	&$0$	&$6$	&$3$	&$-1$	&$3$	&$0$	&$0$	&$2$	&$-1$	&$-1$	&$1$	&$0$	&$1$	&$0$	&$0$	&$-1$	&$-1$	&$0$	&$0$	&$-1$	&$-1$	&$0$	&$0$	\\
$\chi_{15}$&$\circ$&$1035$	&$-21$	&$-5$	&$0$	&$-3$	&$3$	&$3$	&$-1$	&$0$	&$0$	&$1$	&$2b_7$	&$2\overline{b_7}$	&$-1$	&$0$	&$1$	&$0$	&$-1$	&$0$	&$0$	&$0$	&$0$	&$-b_7$	&$-\overline{b_7} $	&$0$	&$0$	\\
$\chi_{16}$&$\circ$&$1035$	&$-21$	&$-5$	&$0$	&$-3$	&$3$	&$3$	&$-1$	&$0$	&$0$	&$1$	&$2\overline{b_7}$	&$2b_7$	&$-1$	&$0$	&$1$	&$0$	&$-1$	&$0$	&$0$	&$0$	&$0$	&$-\overline{b_7}$	&$-b_7 $	&$0$	&$0$	\\
$\chi_{17}$&$+$&$1265$	&$49$	&$-15$	&$5$	&$8$	&$-7$	&$1$	&$-3$	&$0$	&$1$	&$0$	&$-2$	&$-2$	&$1$	&$0$	&$0$	&$-1$	&$0$	&$0$	&$0$	&$0$	&$0$	&$1$	&$1$	&$0$	&$0$	\\
$\chi_{18}$&$+$&$1771$	&$-21$	&$11$	&$16$	&$7$	&$3$	&$-5$	&$-1$	&$1$	&$0$	&$-1$	&$0$	&$0$	&$-1$	&$1$	&$0$	&$0$	&$-1$	&$0$	&$0$	&$1$	&$1$	&$0$	&$0$	&$0$	&$0$	\\
$\chi_{19}$&$+$&$2024$	&$8$	&$24$	&$-1$	&$8$	&$8$	&$0$	&$0$	&$-1$	&$-1$	&$0$	&$1$	&$1$	&$0$	&$-1$	&$0$	&$-1$	&$0$	&$1$	&$1$	&$-1$	&$-1$	&$1$	&$1$	&$0$	&$0$	\\
$\chi_{20}$&$+$&$2277$	&$21$	&$-19$	&$0$	&$6$	&$-3$	&$1$	&$-3$	&$-3$	&$0$	&$2$	&$2$	&$2$	&$-1$	&$1$	&$0$	&$0$	&$0$	&$0$	&$0$	&$0$	&$0$	&$-1$	&$-1$	&$0$	&$0$	\\
$\chi_{21}$&$+$&$3312$	&$48$	&$16$	&$0$	&$-6$	&$0$	&$0$	&$0$	&$-3$	&$0$	&$-2$	&$1$	&$1$	&$0$	&$1$	&$1$	&$0$	&$0$	&$-1$	&$-1$	&$0$	&$0$	&$1$	&$1$	&$0$	&$0$	\\
$\chi_{22}$&$+$&$3520$	&$64$	&$0$	&$10$	&$-8$	&$0$	&$0$	&$0$	&$0$	&$-2$	&$0$	&$-1$	&$-1$	&$0$	&$0$	&$0$	&$0$	&$0$	&$1$	&$1$	&$0$	&$0$	&$-1$	&$-1$	&$1$	&$1$	\\
$\chi_{23}$&$+$&$5313$	&$49$	&$9$	&$-15$	&$0$	&$1$	&$-3$	&$-3$	&$3$	&$1$	&$0$	&$0$	&$0$	&$-1$	&$-1$	&$0$	&$1$	&$0$	&$0$	&$0$	&$0$	&$0$	&$0$	&$0$	&$0$	&$0$	\\
$\chi_{24}$&$+$&$5544$	&$-56$	&$24$	&$9$	&$0$	&$-8$	&$0$	&$0$	&$-1$	&$1$	&$0$	&$0$	&$0$	&$0$	&$-1$	&$0$	&$1$	&$0$	&$0$	&$0$	&$-1$	&$-1$	&$0$	&$0$	&$1$	&$1$	\\
$\chi_{25}$&$+$&$5796$	&$-28$	&$36$	&$-9$	&$0$	&$-4$	&$4$	&$0$	&$1$	&$-1$	&$0$	&$0$	&$0$	&$0$	&$1$	&$-1$	&$-1$	&$0$	&$0$	&$0$	&$1$	&$1$	&$0$	&$0$	&$0$	&$0$	\\
$\chi_{26}$&$+$&$10395$	&$-21$	&$-45$	&$0$	&$0$	&$3$	&$-1$	&$3$	&$0$	&$0$	&$0$	&$0$	&$0$	&$1$	&$0$	&$0$	&$0$	&$0$	&$0$	&$0$	&$0$	&$0$	&$0$	&$0$	&$-1$	&$-1$	\\\bottomrule
\end{tabular}
\end{footnotesize}
\end{center}
\end{sidewaystable}

\begin{table}[ht]
\centering
\begin{small}
\caption{{Data for Proposition \ref{pro:va:phig}}\label{tab:da}}
\begin{tabular}{c|c|c|c}
\toprule
\text{Class} & \text{Cycle Structure} & $C_{-g}$ & $D_{g}$ \\
\midrule
1A & $1^{24}$ & $4096$ & $0$ \\
2A & $1^8 \cdot 2^8$ & $0$ & $0$\\
2B & $2^{12}$ & $0$ & $0$ \\
3A & $1^6\cdot 3^6$ & $64$ & $0$ \\
3B & $3^8$ & $16$ & $0$ \\
4A & $2^4\cdot 4^4$ & $0$ &$0$ \\
4B & $1^4\cdot 2^2\cdot 4^4$ &$0$ & $0$\\
4C & $4^6$ &$0$ &$0$ \\
5A & $1^4\cdot 5^4$ & $16$ &$0$ \\
6A & $1^2\cdot 2^2\cdot 3^2\cdot 6^2$ & $0$& $0$ \\
6B & $6^4$ & $0$&$0$ \\
7A & $1^3\cdot 7^3$ & $8$ &$0$ \\
7B & $1^3\cdot 7^3$ & $8$& $0$\\
8A & $1^2\cdot 2\cdot 4\cdot 8^2$ &$0$ &$0$ \\
10A & $2^2\cdot 10^2$ &$0$ &$0$ \\
11A & $1^2\cdot 11^2$ & $4$& $0$\\
12A & $2 \cdot 4 \cdot 6 \cdot 12$ & $0$& $0$\\
12B & $12^2$ & $0$& $-12 i$\\
14A & $1 \cdot 2 \cdot 7 \cdot 14$ & $0$&$0$ \\
14B & $1 \cdot 2 \cdot 7 \cdot 14$ & $0$&$0$ \\
15A & $1 \cdot 3 \cdot 5 \cdot 15$ & $4$& $0$\\
15B & $1 \cdot 3 \cdot 5 \cdot 15$ & $4$& $0$\\
21A & $3 \cdot 21$ & $2$&$3\sqrt{7} i$ \\
21B & $3 \cdot 21$ & $2$& $3\sqrt{7} i$\\
23A & $1 \cdot 23$ & $2 $& $\sqrt{23}  i$\\
23B & $1 \cdot 23$ & $2$& $\sqrt{23}  i$\\
\bottomrule
\end{tabular}
\end{small}
\end{table} 

\clearpage

\section{Multiplicities}\label{sec:mlt}

\vfill

\hfill
\begin{minipage}[t]{0.27\textwidth}
	\centering
\begin{footnotesize}
\begin{tabular}{c|c}
	\toprule
	$j$& $\rho_j(\chi_1)$ \\
	\midrule
	$1$ & $-1$ \\
	$2$ & $0$ \\
	$3$ & $0$ \\
	$4$ & $0$ \\
	$5$ & $0$ \\
	$6$ & $1$ \\
	$7$ & $4$ \\
	$8$ & $32$ \\
	$9$ & $588$ \\
	$10$ & $10984$ \\
	$11$ & $213361$ \\
	$12$ & $4272898$ \\
	$13$ & $86530367$ \\
	$14$ & $1763550556$ \\
	$15$& $36133233594$ \\
	$16$& $743689742272$ \\
	$17$& $15366803399428$ \\
	$18$& $318626547565247$ \\
	$19$& $6627096180118217$ \\
\bottomrule
\end{tabular}
\end{footnotesize}
\end{minipage}
\hfill
\begin{minipage}[t]{0.27\textwidth}
\centering
\begin{footnotesize}
\begin{tabular}{c|c}
	\toprule
	$j$& $\rho_j(\chi_2)$ \\
\midrule
	$1$ & $1$ \\
		$2$ & $0$ \\
		$3$ & $0$ \\
		$4$ & $0$ \\
		$5$ & $1$ \\
		$6$ & $7$ \\
		$7$ & $50$ \\
		$8$ & $700$ \\
		$9$ & $12718$ \\
		$10$ & $246230$ \\
		$11$ & $4886508$ \\
		$12$ & $98209502$ \\
		$13$ & $1989650854$ \\
		$14$ & $40558083580$ \\
		$15$ & $831048880350$ \\
		$16$& $17104793197688$ \\
	$17$& $353436020602096$ \\
	$18$& $7328407831026159$ \\
	$19$& $152423198327490650$ \\
\bottomrule
\end{tabular}
\end{footnotesize}
\end{minipage}
\hfill
\begin{minipage}[t]{0.27\textwidth}
\centering
\begin{footnotesize}
\begin{tabular}{c|c}
	\toprule 	$j$& $\rho_j(\chi_3)$ \\
		\midrule	$1$ & $0$ \\
		$2$ & $0$ \\
		$3$ & $0$ \\
		$4$ & $0$ \\
		$5$ & $0$ \\
		$6$ & $3$ \\
		$7$ & $60$ \\
		$8$ & $1202$ \\
		$9$ & $24073$ \\
		$10$ & $477804$ \\
		$11$ & $9540338$ \\
		$12$ & $192043022$ \\
		$13$ & $3892241220$ \\
		$14$ & $79349833252$ \\
		$15$ & $1625949221980$ \\
			$16$& $33465812442916$ \\
			$17$& $691504782811080$ \\
			$18$& $14338186634603811$ \\
			$19$& $298219286691924780$ \\
		\bottomrule
\end{tabular}
\end{footnotesize}
\end{minipage}

\vfill

\begin{minipage}[t]{0.27\textwidth}
	\centering
	\begin{footnotesize}
	\begin{tabular}{c|c}
	\toprule
	$j$& $\rho_j(\chi_4)$ \\
	\midrule
	$1$ & $0$ \\
		$2$ & $0$ \\
		$3$ & $0$ \\
		$4$ & $0$ \\
		$5$ & $0$ \\
		$6$ & $3$ \\
		$7$ & $60$ \\
		$8$ & $1202$ \\
		$9$ & $24073$ \\
		$10$ & $477804$ \\
		$11$ & $9540338$ \\
		$12$ & $192043022$ \\
		$13$ & $3892241220$ \\
		$14$ & $79349833252$ \\
		$15$ & $1625949221980$ \\
		$16$& $33465812442916$ \\
			$17$& $691504782811080$ \\
			$18$& $14338186634603811$ \\
			$19$& $298219286691924780$ \\
	\bottomrule
\end{tabular}
\end{footnotesize}
\end{minipage}
\hfill
\begin{minipage}[t]{0.27\textwidth}
	\centering
	\begin{footnotesize}
	\begin{tabular}{c|c}
	\toprule
	$j$& $\rho_j(\chi_5)$ \\
\midrule
	$1$ & $0$ \\
		$2$ & $0$ \\
		$3$ & $0$ \\
		$4$ & $0$ \\
		$5$ & $1$ \\
		$6$ & $17$ \\
		$7$ & $338$ \\
		$8$ & $6432$ \\
		$9$ & $124468$ \\
		$10$ & $2455518$ \\
		$11$ & $48998211$ \\
		$12$ & $985985635$ \\
		$13$ & $19980883491$ \\
		$14$ & $407332398340$ \\
		$15$ & $8346560410940$ \\
		$16$& $171791297467968$ \\
			$17$& $3549725184263524$ \\
			$18$& $73602694645974787$ \\
			$19$& $1530859024258691429$ \\
	\bottomrule
\end{tabular}
\end{footnotesize}
\end{minipage}
\hfill
\begin{minipage}[t]{0.27\textwidth}
	\centering
	\begin{footnotesize}
	\begin{tabular}{c|c}
	\toprule 	$j$& $\rho_j(\chi_6)$ \\
		\midrule
		$1$ & $0$ \\
		$2$ & $0$ \\
		$3$ & $0$ \\
		$4$ & $0$ \\
		$5$ & $1$ \\
		$6$ & $17$ \\
		$7$ & $338$ \\
		$8$ & $6432$ \\
		$9$ & $124468$ \\
		$10$ & $2455518$ \\
		$11$ & $48998211$ \\
		$12$ & $985985635$ \\
		$13$ & $19980883491$ \\
		$14$ & $407332398340$ \\
		$15$ & $8346560410940$ \\
		$16$& $171791297467968$ \\
			$17$& $3549725184263524$ \\
			$18$& $73602694645974787$ \\
			$19$& $1530859024258691429$ \\
		\bottomrule
\end{tabular}
\end{footnotesize}
\end{minipage}

\vfill

\newpage

\vfill

\begin{minipage}[t]{0.27\textwidth}
	\centering
	\begin{footnotesize}
	\begin{tabular}{c|c}
	\toprule 	$j$& $\rho_j(\chi_7)$ \\
		\midrule	$1$ & $0$ \\
		$2$ & $1$ \\
		$3$ & $0$ \\
		$4$ & $0$ \\
		$5$ & $4$ \\
		$6$ & $37$ \\
		$7$ & $416$ \\
		$8$ & $7110$ \\
		$9$ & $136764$ \\
		$10$ & $2685594$ \\
		$11$ & $53477314$ \\
		$12$ & $1075713141$ \\
		$13$ & $21798003732$ \\
		$14$ & $444366896180$ \\
		$15$ & $9105358123214$ \\
		$16$& $187408781333964$ \\
	$17$& $3872428053060448$ \\
	$18$& $80293852135947487$ \\
	$19$& $1670028044008883948$ \\
	\bottomrule
\end{tabular}
\end{footnotesize}
\end{minipage}
\hfill
\begin{minipage}[t]{0.27\textwidth}
\centering
\begin{footnotesize}
\begin{tabular}{c|c}
	\toprule 	$j$& $\rho_j(\chi_8)$ \\
		\midrule	$1$ & $0$ \\
		$2$ & $0$ \\
		$3$ & $0$ \\
		$4$ & $1$ \\
		$5$ & $3$ \\
		$6$ & $25$ \\
		$7$ & $396$ \\
		$8$ & $7142$ \\
		$9$ & $136714$ \\
		$10$ & $2691192$ \\
		$11$ & $53672840$ \\
		$12$ & $1079927072$ \\
		$13$ & $21884021444$ \\
		$14$ & $446126991808$ \\
		$15$ & $9141476263728$ \\
		$16$& $188152401661324$ \\
	$17$& $3887794404057256$ \\
	$18$& $80612475938478939$ \\
	$19$& $1676655126438353900$ \\
	\bottomrule
\end{tabular}
\end{footnotesize}
\end{minipage}
\hfill
\begin{minipage}[t]{0.27\textwidth}
\centering
\begin{footnotesize}
\begin{tabular}{c|c}
	\toprule 	$j$& $\rho_j(\chi_9)$ \\
		\midrule	$1$ & $0$ \\
		$2$ & $0$ \\
		$3$ & $0$ \\
		$4$ & $0$ \\
		$5$ & $5$ \\
		$6$ & $54$ \\
		$7$ & $754$ \\
		$8$ & $13542$ \\
		$9$ & $261232$ \\
		$10$ & $5141116$ \\
		$11$ & $102475525$ \\
		$12$ & $2061698776$ \\
		$13$ & $41778887223$ \\
		$14$ & $851699294520$ \\
		$15$ & $17451918534212$ \\
			$16$& $359200078801932$ \\
			$17$& $7422153237323972$ \\
			$18$& $153896546781922274$ \\
			$19$& $3200887068267575377$ \\
		\bottomrule
\end{tabular}
\end{footnotesize}
\end{minipage}

\vfill

\begin{minipage}[t]{0.27\textwidth}
	\centering
	\begin{footnotesize}
	\begin{tabular}{c|c}
	\toprule 	$j$& $\rho_j(\chi_{10})$ \\
		\midrule	$1$ & $0$ \\
		$2$ & $0$ \\
		$3$ & $0$ \\
		$4$ & $0$ \\
		$5$ & $3$ \\
		$6$ & $58$ \\
		$7$ & $1090$ \\
		$8$ & $21114$ \\
		$9$ & $413615$ \\
		$10$ & $8179718$ \\
		$11$ & $163288225$ \\
		$12$ & $3286372442$ \\
		$13$ & $66601777885$ \\
		$14$ & $1357768898140$ \\
		$15$ & $27821833170578$ \\
		$16$& $572637455308212$ \\
			$17$& $11832416229590740$ \\
			$18$& $245342309919677533$ \\
			$19$& $5102863382164608875$ \\
		\bottomrule
\end{tabular}
\end{footnotesize}
\end{minipage}
\hfill
\begin{minipage}[t]{0.27\textwidth}
\centering
\begin{footnotesize}
\begin{tabular}{c|c}
	\toprule 	$j$& $\rho_j(\chi_{11})$ \\
		\midrule	$1$ & $0$ \\
		$2$ & $0$ \\
		$3$ & $0$ \\
		$4$ & $0$ \\
		$5$ & $3$ \\
		$6$ & $58$ \\
		$7$ & $1090$ \\
		$8$ & $21114$ \\
		$9$ & $413615$ \\
		$10$ & $8179718$ \\
		$11$ & $163288225$ \\
		$12$ & $3286372442$ \\
		$13$ & $66601777885$ \\
		$14$ & $1357768898140$ \\
		$15$ & $27821833170578$ \\
		$16$& $572637455308212$ \\
	$17$& $11832416229590740$ \\
	$18$& $245342309919677533$ \\
	$19$& $5102863382164608875$ \\
\bottomrule
\end{tabular}
\end{footnotesize}
\end{minipage}
\hfill
\begin{minipage}[t]{0.27\textwidth}
	\centering
	\begin{footnotesize}
	\begin{tabular}{c|c}
	\toprule 	$j$& $\rho_j(\chi_{12})$ \\
		\midrule	$1$ & $0$ \\
		$2$ & $0$ \\
		$3$ & $0$ \\
		$4$ & $0$ \\
		$5$ & $3$ \\
		$6$ & $67$ \\
		$7$ & $1393$ \\
		$8$ & $27200$ \\
		$9$ & $531678$ \\
		$10$ & $10514976$ \\
		$11$ & $209940945$ \\
		$12$ & $4225359670$ \\
		$13$ & $85630844805$ \\
		$14$ & $1745702303864$ \\
		$15$ & $35770928205528$ \\
		$16$& $736248166334080$ \\
	$17$& $15213106579051620$ \\
	$18$& $315440112551329871$ \\
	$19$& $6560824348474289955$ \\
\bottomrule
\end{tabular}
\end{footnotesize}
\end{minipage}

\vfill

\newpage

\vfill

\begin{minipage}[t]{0.27\textwidth}
	\centering
	\begin{footnotesize}
	\begin{tabular}{c|c}
	\toprule 	$j$& $\rho_j(\chi_{13})$ \\
		\midrule	$1$ & $0$ \\
		$2$ & $0$ \\
		$3$ & $0$ \\
		$4$ & $0$ \\
		$5$ & $3$ \\
		$6$ & $67$ \\
		$7$ & $1393$ \\
		$8$ & $27200$ \\
		$9$ & $531678$ \\
		$10$ & $10514976$ \\
		$11$ & $209940945$ \\
		$12$ & $4225359670$ \\
		$13$ & $85630844805$ \\
		$14$ & $1745702303864$ \\
		$15$ & $35770928205528$ \\
			$16$& $736248166334080$ \\
			$17$& $15213106579051620$ \\
			$18$& $315440112551329871$ \\
			$19$& $6560824348474289955$ \\
	\bottomrule
\end{tabular}
\end{footnotesize}
\end{minipage}
\hfill
\begin{minipage}[t]{0.27\textwidth}
	\centering
	\begin{footnotesize}
	\begin{tabular}{c|c}
	\toprule 	$j$& $\rho_j(\chi_{14})$ \\
		\midrule	$1$ & $0$ \\
		$2$ & $0$ \\
		$3$ & $0$ \\
		$4$ & $0$ \\
		$5$ & $6$ \\
		$6$ & $96$ \\
		$7$ & $1526$ \\
		$8$ & $28562$ \\
		$9$ & $557676$ \\
		$10$ & $11006652$ \\
		$11$ & $219534788$ \\
		$12$ & $4417619232$ \\
		$13$ & $89524623990$ \\
		$14$ & $1825061766808$ \\
		$15$ & $37396922706972$ \\
			$16$& $769714201094116$ \\
			$17$& $15904612719070560$ \\
			$18$& $329778307146655584$ \\
			$19$& $6859043676418159530$ \\
	\bottomrule
\end{tabular}
\end{footnotesize}
\end{minipage}
\hfill
\begin{minipage}[t]{0.27\textwidth}
	\centering
	\begin{footnotesize}
	\begin{tabular}{c|c}
	\toprule 	$j$& $\rho_j(\chi_{15})$ \\
		\midrule	$1$ & $0$ \\
		$2$ & $0$ \\
		$3$ & $0$ \\
		$4$ & $0$ \\
		$5$ & $3$ \\
		$6$ & $71$ \\
		$7$ & $1453$ \\
		$8$ & $28402$ \\
		$9$ & $555772$ \\
		$10$ & $10992780$ \\
		$11$ & $219481283$ \\
		$12$ & $4417403042$ \\
		$13$ & $89523086025$ \\
		$14$ & $1825052137116$ \\
		$15$ & $37396877433580$ \\
			$16$& $769713978776996$ \\
			$17$& $15904611361862700$ \\
			$18$& $329778299186044709$ \\
			$19$& $6859043635166214735$ \\
	\bottomrule
\end{tabular}
\end{footnotesize}
\end{minipage}

\vfill

\begin{minipage}[t]{0.27\textwidth}
	\centering
	\begin{footnotesize}
	\begin{tabular}{c|c}
	\toprule 	$j$& $\rho_j(\chi_{16})$ \\
		\midrule	$1$ & $0$ \\
		$2$ & $0$ \\
		$3$ & $0$ \\
		$4$ & $0$ \\
		$5$ & $3$ \\
		$6$ & $71$ \\
		$7$ & $1453$ \\
		$8$ & $28402$ \\
		$9$ & $555772$ \\
		$10$ & $10992780$ \\
		$11$ & $219481283$ \\
		$12$ & $4417403042$ \\
		$13$ & $89523086025$ \\
		$14$ & $1825052137116$ \\
		$15$ & $37396877433580$ \\
			$16$& $769713978776996$ \\
			$17$& $15904611361862700$ \\
			$18$& $329778299186044709$ \\
			$19$& $6859043635166214735$ \\
	\bottomrule
\end{tabular}
\end{footnotesize}
\end{minipage}
\hfill
\begin{minipage}[t]{0.27\textwidth}
	\centering
	\begin{footnotesize}
	\begin{tabular}{c|c}
	\toprule 	$j$& $\rho_j(\chi_{17})$ \\
		\midrule	$1$ & $0$ \\
		$2$ & $0$ \\
		$3$ & $0$ \\
		$4$ & $1$ \\
		$5$ & $9$ \\
		$6$ & $112$ \\
		$7$ & $1896$ \\
		$8$ & $35252$ \\
		$9$ & $682310$ \\
		$10$ & $13452722$ \\
		$11$ & $268338955$ \\
		$12$ & $5399473660$ \\
		$13$ & $109419507235$ \\
		$14$ & $2230632665932$ \\
		$15$ & $45707365183430$ \\
			$16$& $940761907087336$ \\
			$17$& $19438971554928340$ \\
			$18$& $403062377450532250$ \\
			$19$& $8383275618279671045$ \\
	\bottomrule
\end{tabular}
\end{footnotesize}
\end{minipage}
\hfill
\begin{minipage}[t]{0.27\textwidth}
	\centering
	\begin{footnotesize}
	\begin{tabular}{c|c}
	\toprule 	$j$& $\rho_j(\chi_{18})$ \\
		\midrule	$1$ & $0$ \\
		$2$ & $0$ \\
		$3$ & $0$ \\
		$4$ & $1$ \\
		$5$ & $8$ \\
		$6$ & $133$ \\
		$7$ & $2530$ \\
		$8$ & $48702$ \\
		$9$ & $951772$ \\
		$10$ & $18815382$ \\
		$11$ & $375576143$ \\
		$12$ & $7558734652$ \\
		$13$ & $153184456727$ \\
		$14$ & $3122870255064$ \\
		$15$ & $63990226945984$ \\
			$16$& $1317066208822332$ \\
			$17$& $27214557646036068$ \\
			$18$& $564287314494151039$ \\
			$19$& $11736585788620199217$ \\
	\bottomrule
\end{tabular}
\end{footnotesize}
\end{minipage}

\vfill

\newpage 

\vfill

\begin{minipage}[t]{0.27\textwidth}
	\centering
	\begin{footnotesize}
	\begin{tabular}{c|c}
	\toprule 	$j$& $\rho_j(\chi_{19})$ \\
		\midrule	$1$ & $0$ \\
		$2$ & $0$ \\
		$3$ & $0$ \\
		$4$ & $0$ \\
		$5$ & $9$ \\
		$6$ & $162$ \\
		$7$ & $2914$ \\
		$8$ & $55730$ \\
		$9$ & $1088766$ \\
		$10$ & $21510644$ \\
		$11$ & $429262372$ \\
		$12$ & $8638706004$ \\
		$13$ & $175068938428$ \\
		$14$ & $3569000520108$ \\
		$15$ & $73131717678906$ \\
			$16$& $1505218677685924$ \\
			$17$& $31102352494722752$ \\
			$18$& $644899793150420208$ \\
			$19$& $13413240928712331268$ \\
		\bottomrule
\end{tabular}
\end{footnotesize}
\end{minipage}
\hfill
\begin{minipage}[t]{0.27\textwidth}
	\centering
	\begin{footnotesize}
	\begin{tabular}{c|c}
	\toprule 	$j$& $\rho_j(\chi_{20})$ \\
		\midrule	$1$ & $0$ \\
		$2$ & $0$ \\
		$3$ & $0$ \\
		$4$ & $1$ \\
		$5$ & $11$ \\
		$6$ & $176$ \\
		$7$ & $3298$ \\
		$8$ & $62954$ \\
		$9$ & $1225364$ \\
		$10$ & $24199272$ \\
		$11$ & $482933730$ \\
		$12$ & $9718667200$ \\
		$13$ & $196952942406$ \\
		$14$ & $4015126719456$ \\
		$15$ & $82273193736660$ \\
			$16$& $1693371092666644$ \\
			$17$& $34990146896188944$ \\
			$18$& $725512268805550800$ \\
			$19$& $15089896055118395130$ \\
		\bottomrule
\end{tabular}
\end{footnotesize}
\end{minipage}
\hfill
\begin{minipage}[t]{0.27\textwidth}
	\centering
	\begin{footnotesize}
	\begin{tabular}{c|c}
	\toprule 	$j$& $\rho_j(\chi_{21})$ \\
		\midrule	$1$ & $0$ \\
		$2$ & $0$ \\
		$3$ & $0$ \\
		$4$ & $1$ \\
		$5$ & $17$ \\
		$6$ & $276$ \\
		$7$ & $4824$ \\
		$8$ & $91516$ \\
		$9$ & $1783082$ \\
		$10$ & $35205924$ \\
		$11$ & $702468518$ \\
		$12$ & $14136287112$ \\
		$13$ & $286477566396$ \\
		$14$ & $5840188486264$ \\
		$15$ & $119670116455776$ \\
			$16$& $2463085293760760$ \\
			$17$& $50894759615259504$ \\
			$18$& $1055290575952428732$ \\
			$19$& $21948939731536554660$ \\
		\bottomrule
\end{tabular}
\end{footnotesize}
\end{minipage}

\vfill

\begin{minipage}[t]{0.27\textwidth}
	\centering
	\begin{footnotesize}
	\begin{tabular}{c|c}
	\toprule 	$j$& $\rho_j(\chi_{22})$ \\
		\midrule	$1$ & $0$ \\
		$2$ & $0$ \\
		$3$ & $1$ \\
		$4$ & $2$ \\
		$5$ & $20$ \\
		$6$ & $298$ \\
		$7$ & $5160$ \\
		$8$ & $97456$ \\
		$9$ & $1895723$ \\
		$10$ & $37419312$ \\
		$11$ & $746601020$ \\
		$12$ & $15024171162$ \\
		$13$ & $304469346620$ \\
		$14$ & $6206965644824$ \\
		$15$ & $127185643497524$ \\
			$16$& $2617771882979168$ \\
			$17$& $54091049236505680$ \\
			$18$& $1121564865256303860$ \\
			$19$& $23327375571282983780$ \\
		\bottomrule
\end{tabular}
\end{footnotesize}
\end{minipage}
\hfill
\begin{minipage}[t]{0.27\textwidth}
	\centering
	\begin{footnotesize}
	\begin{tabular}{c|c}
	\toprule 	$j$& $\rho_j(\chi_{23})$ \\
		\midrule	$1$ & $0$ \\
		$2$ & $0$ \\
		$3$ & $0$ \\
		$4$ & $1$ \\
		$5$ & $23$ \\
		$6$ & $420$ \\
		$7$ & $7676$ \\
		$8$ & $146628$ \\
		$9$ & $2859014$ \\
		$10$ & $56467762$ \\
		$11$ & $1126842900$ \\
		$12$ & $22676817738$ \\
		$13$ & $459556836504$ \\
		$14$ & $9368630130332$ \\
		$15$ & $191970785048182$ \\
			$16$& $3951199191568392$ \\
			$17$& $81643676086789096$ \\
			$18$& $1692861960988474434$ \\
			$19$& $35209757461889105240$ \\
		\bottomrule
\end{tabular}
\end{footnotesize}
\end{minipage}
\hfill
\begin{minipage}[t]{0.27\textwidth}
	\centering
	\begin{footnotesize}
	\begin{tabular}{c|c}
	\toprule 	$j$& $\rho_j(\chi_{24})$ \\
		\midrule	$1$ & $0$ \\
		$2$ & $0$ \\
		$3$ & $0$ \\
		$4$ & $1$ \\
		$5$ & $21$ \\
		$6$ & $409$ \\
		$7$ & $7880$ \\
		$8$ & $152370$ \\
		$9$ & $2979352$ \\
		$10$ & $58899976$ \\
		$11$ & $1175720712$ \\
		$12$ & $23662170501$ \\
		$13$ & $479534183808$ \\
		$14$ & $9775942922756$ \\
		$15$ & $200317240431118$ \\
			$16$& $4122989920977444$ \\
			$17$& $85193398112007472$ \\
			$18$& $1766464638092399079$ \\
			$19$& $36740616389990128928$ \\
		\bottomrule
\end{tabular}
\end{footnotesize}
\end{minipage}

\vfill

\newpage

\vfill

\begin{minipage}[t]{0.27\textwidth}
	\centering
	\begin{footnotesize}
	\begin{tabular}{c|c}
	\toprule 	$j$& $\rho_j(\chi_{25})$ \\
		\midrule	$1$ & $0$ \\
		$2$ & $0$ \\
		$3$ & $0$ \\
		$4$ & $0$ \\
		$5$ & $21$ \\
		$6$ & $437$ \\
		$7$ & $8260$ \\
		$8$ & $159366$ \\
		$9$ & $3115758$ \\
		$10$ & $61584248$ \\
		$11$ & $1229193580$ \\
		$12$ & $24737868955$ \\
		$13$ & $501332135142$ \\
		$14$ & $10220309637244$ \\
		$15$ & $209422597930386$ \\
			$16$& $4310398700098764$ \\
			$17$& $89065826157294728$ \\
			$18$& $1846758490201103001$ \\
			$19$& $38410644433902142762$ \\
		\bottomrule
\end{tabular}
\end{footnotesize}
\end{minipage}
\hspace{1em}
\begin{minipage}[t]{0.27\textwidth}
	\centering
	\begin{footnotesize}
	\begin{tabular}{c|c}
	\toprule 	$j$& $\rho_j(\chi_{26})$ \\
		\midrule	$1$ & $0$ \\
		$2$ & $0$ \\
		$3$ & $0$ \\
		$4$ & $3$ \\
		$5$ & $42$ \\
		$6$ & $774$ \\
		$7$ & $14880$ \\
		$8$ & $286494$ \\
		$9$ & $5589430$ \\
		$10$ & $110450340$ \\
		$11$ & $2204567190$ \\
		$12$ & $44367133544$ \\
		$13$ & $899129253330$ \\
		$14$ & $18329906061720$ \\
		$15$ & $375594904657420$ \\
			$16$& $7730606559774972$ \\
			$17$& $159737623830269520$ \\
			$18$& $3312121208966044860$ \\
			$19$& $68888655803361851310$ \\
		\bottomrule
\end{tabular}
\end{footnotesize}
\end{minipage}

\clearpage



\end{document}